\documentclass{siamart171218}

\usepackage{amsmath}
\usepackage{amsfonts}
\usepackage{graphicx}
\usepackage{epstopdf}
\usepackage{algorithmic}
\ifpdf
  \DeclareGraphicsExtensions{.eps,.pdf,.png,.jpg}
\else
  \DeclareGraphicsExtensions{.eps}
\fi

\newsiamremark{remark}{Remark}
\newsiamremark{hypothesis}{Hypothesis}
\crefname{hypothesis}{Hypothesis}{Hypotheses}
\newsiamthm{claim}{Claim}

\headers{Distributed filtered hyperinterpolation}{S.-B. Lin, Y. G. Wang, and D.-X. Zhou}

\title{Distributed filtered hyperinterpolation for noisy data on the sphere\thanks{\today.
\funding{The first author acknowledges support from the National Natural Science Foundation of China, Grant No. 61876133. The second author acknowledges support from the Australian Research Council under Discovery Project DP180100506. The third author acknowledges support from the Research Grant Council of Hong Kong, Project No. CityU 11338616. This material is based upon work supported by the National Science Foundation under Grant No. DMS-1439786 while the second author were in residence at the Institute for Computational and Experimental Research in Mathematics in Providence, RI, during the Point Configurations in Geometry, Physics and Computer Science program.}}}

\author{Shao-bo Lin\thanks{School of Management, Xi'an Jiaotong University, Xi'an, China
  (\email{sblin1983@gmail.com}).}
\and Yu Guang Wang\thanks{School of Mathematics and Statistics, University of New South Wales, Sydney, NSW, Australia 
  (\email{yuguang.wang@unsw.edu.au}).}
\and Ding-Xuan Zhou\thanks{School of Data Science and Department of Mathematics, 
City University of Hong Kong, Tat Chee Avenue, Kowloon, Hong Kong 
  (\email{mazhou@cityu.edu.hk}).}}

\usepackage{amsopn}

\newcommand{\Rone}{ % real numbers
    \mathbb{R}
}

\newcommand{\sfgrad}[1][]{ % surface-gradient
	\nabla_{*}
}

\newcommand{\sfcurl}[1][]{ % surface-gradient
	\mathbf{L}
}

\newcommand{\Jw}[1][\alpha,\beta]{ % Jacobi weight w_{\alpha,\beta}
w_{#1}
}

\newcommand{\sph}[1]{ % sphere S^d
    \mathbb{S}^{#1}
}

\newcommand{\imat}[1][d]{ % identity matrix
    I
}

\newcommand{\Lpw}[2][\Jw]{ % Sobolev spaces W_p^s for w_{\alpha,\beta}
\mathbb{L}_{#2}(#1)
}

\newcommand{\InnerLGb}[2][{\Jw[r-\frac{1}{2},r-\frac{1}{2}]}]{ % inner product for Gegenbauer weights
\left(#2\right)_{\Lpw[{#1}]{2}}
}

\newcommand{\Diff}[2][t]{ % derivative
\ifthenelse{\equal{#2}{1}}{\frac{\mathrm{d}}{\mathrm{d}#1}}{
\left(\frac{\mathrm{d}}{\mathrm{d}#1}\right)^{#2}}
}

\newcommand{\bigo}[2] { % big O notation
    \mathcal{O}_{#1}\left(#2\right)
}

\newcommand{\fWend}[1]{ % Wendland function
    \widetilde{\phi}_{#1}
}

\newcommand{\fnWend}[1]{ % Wendland function
    \phi_{#1}
}

\newcommand{\bx}{{\bf x}}

\newcommand{\bz}{{\bf z}}

\ifpdf
\hypersetup{
  pdftitle={Distributed Learning for},
  pdfauthor={S.-B. Lin, Y. G. Wang, and D.-X. Zhou}
}
\fi

\begin{document}

\maketitle

% REQUIRED
\begin{abstract}
  Problems in astrophysics, space weather research and geophysics usually need to analyze noisy big data on the sphere.  This paper develops distributed filtered hyperinterpolation for noisy data on the sphere, which assigns the data fitting task to multiple servers to find a good approximation of the mapping of input and output data. For each server, the approximation is a filtered hyperinterpolation on the sphere by a small proportion of quadrature nodes.
  The distributed strategy allows parallel computing for data processing and model selection and thus reduces computational cost for each server while preserves the approximation capability compared to the filtered hyperinterpolation. We prove quantitative relation between the approximation capability of distributed filtered hyperinterpolation and the numbers of input data and servers. Numerical examples show the efficiency and accuracy of the proposed method.  
\end{abstract}

% REQUIRED

% REQUIRED

%%%%%%%%%%%%%%%%%%%%%%%%%%%%%%%%%%%%%%%%%%%%%%%%%%%%
\section{Introduction}
In cosmic microwave background analysis, global ionospheric prediction to geomagnetic storms, climate change modelling, environment governance and meteorology and remote sensing, data are collected on the sphere and usually big and noisy \cite{Planck2018I,FrGeSc1998,LeSlWoWa2019,Mendillo2006,MuPaAnPa2013,Planck2016A1}.
One of critical tasks of big data analysis on the sphere is to find an effective data fitting strategy to approximate the mapping between input and output data. There have been many useful methods for fitting spherical data, for example approximations by spherical harmonics \cite{Muller1966}, spherical basis functions \cite{LeNaWaWe2006,LeSlWe2010,HeSlWo2017,NaSuWaWe2007}, spherical wavelets \cite{FrGeSc1998}, spherical needlets \cite{BaKeMaPi2009,NaPeWa2006,FaPaLe2018,KePhPi2011,Marinucci_etal2008,WaLeSlWo2017}, spherical kernel methods \cite{FaPaLeMa2018,Lin2018} and spherical filtered hyperinterpolation \cite{SlWo2012}.
When noise is sufficiently small and decreases with size of data, least squares regularization can be used to reduce noise in learning representation, see e.g. \cite{LeNaWaWe2006,HeSlWo2017}. 
This method is however not suitable when the size of noisy data is large, as then the regularization condition implies that noise must be close to zero. In this paper, we propose a new strategy based on distributed learning --- \emph{distributed filtered hyperintepolation}, which assigns the data fitting task to multiple servers then synthesizes them as a global prediction model.

Spherical filtered hyperinterpolation, developed by Sloan and Womersley \cite{Sloan1995,SlWo2012}, is a constructive approach: given degree $L$ (which indicates the level of precision), use $N$ data on the sphere to find a filtered expansion of spherical harmonics up to degree $2L$. The filtering strategy uses an appropriate filter and the data at nodes of a quadrature rule. When the quadrature rule that is a set of pairs of points on the sphere and real weights is exact for numerical integration of polynomials up to degree $(c+1)L$ (see the definition of \eqref{eq:quadrature}), the filtered hyperinterpolation reaches the best polynomial approximation \cite{SlWo2012,WaSl2017}. The computational cost is thus determined by the number $N$ of data, which has at least order $\bigo{}{L^2}$ \cite{HeSlWo2010}. 
The cost becomes heavy as degree $L$ or the number of data increases. 
One way to reduce the computational burden is to distribute the approximation task to multiple servers, each of which works on a fraction of the total computation using a small proportion of all data, and then synthesize the computed fitting models of all servers to produce a global predictor. To achieve that, we apply \emph{distributed learning} strategy \cite{LiGuZh2017,LiZh2018} to spherical filtered hyperinterpolation, which leads to distributed filtered hyperinterpolation (DFH).

%The performance of distributed learning has been proved good for big data processing on Euclidean space, as verified by many algorithms such as Tikhonov regularization \cite{LiGuZh2017}, kernel-based
%gradient descent \cite{LiZh2018}, local average regression
%\cite{ChLiWa2017} and kernel-based spectral algorithms
%\cite{GuLiZh2017}.
 
The proposed DFH can fit $N$ noisy data $y_i=f^*(\mathbf x_i)+\epsilon_i$, $i=1,\dots,N$,
for continuous function $f^*$ on the sphere, and independent bounded noises $\epsilon_i$. We show that the approximation error of DFH for such noisy data depends on the number of data and the smoothness of the target function $f^*$. 
We also consider the DFH with independent random interpolation points, which we call \emph{distributed filtered hyperinterpolation with random sampling}, and prove that for distributions of the random points satisfying appropriate conditions, DFH has the same approximation capability as that with ``deterministic'' quadrature rule.

The rest of the paper is organized as follows. In Section~\ref{Sec.Filtered kernel}, we introduce filtered kernel, quadrature rule and ``ordinary'' spherical filtered hyperinterpolation. In Section~\ref{Sec.Fixed design}, we define distributed filtered hyperinterpolation with ``deterministic'' quadrature rule and give the relation of its approximation error and numbers of data and servers. Section~\ref{Sec.Random} defines the distributed filtered hyperinterpolation with random sampling and gives its error estimate. Section~\ref{sec:numer.exam} gives the numerical examples of the distributed filtered hyperinterpolation where we study the impact of the numbers of data, servers and noise on approximation error. Section~\ref{Sec.Proof} gives the proofs for the main results.

\section{Filtered Kernel, Quadrature Rule and Filtered Hyperinterpolation}\label{Sec.Filtered kernel}
In this section, we introduce the spherical filtered
hyperinterpolation defined by filtered kernels and quadrature
rule.

\subsection{Filtered  kernels}

 For $d\ge2$, let $\mathbf x\cdot\mathbf y$ be the inner product of
two points $\mathbf x,\mathbf y$ in $\mathbb R^{d+1}$, and the
Euclidean norm $|\mathbf x|:=\sqrt{\mathbf x\cdot\mathbf x}$. Let
$\mathbb S^d:=\{\mathbf x\in\mathbb R^{d+1}: |\mathbf x|=1\}$  be
the unit sphere of $\mathbb R^{d+1}$. The $\mathbb
S^d$ is a compact metric space with geodesic distance
\begin{equation*}
    \mbox{dist}(\mathbf x,\mathbf y):=\arccos(\mathbf x\cdot\mathbf y), \quad \mathbf
    x,\mathbf y\in\mathbb S^d
\end{equation*}
as the metric. For $1\leq p<\infty$, let $L_p(\mathbb S^d)$ be
the real-valued $L_p$ space on $\mathbb S^d$ with Lebesgue
measure $\omega:=\omega_d$ and $L_p$ norm $\|f\|_{L_p(\mathbb
S^d)}:=\left(\int_{\mathbb S^d}|f(\mathbf x)|^p\mathrm{d}\omega(\mathbf
x)\right)^{1/p}$ for $f\in L_p(\mathbb S^d)$. In particular,
$L_2(\mathbb S^d)$ is a Hilbert space with inner product
\begin{equation*}
    \langle f,g\rangle_{L_2(\mathbb S^d)}:=\int_{\mathbb S^d}
    f(\mathbf x)g(\mathbf x)\mathrm{d}\omega(\mathbf
       x), \quad f,g\in L_p(\mathbb S^d).
\end{equation*}
Denote by $L_{\infty}(\mathbb S^d)$ the space of all real-valued
continuous functions on $\mathbb S^d$ with uniform norm
$\|f\|_{L_\infty}:=\max_{\mathbf x\in\mathbb S^d} |f(\mathbf x)|$.
The volume of $\mathbb S^d$ is
$$
    |\mathbb S^d|:= \omega_d(\mathbb S^d) := \frac{2\pi^{\frac{d+1}{2}}}{\Gamma(\frac{d+1}{2})}.
$$

For $\ell\in\mathbb N_0:=\{0,1,\dots\}$, the restriction to $\mathbb S^d$ of a homogeneous harmonic polynomial of degree $\ell$ is called a spherical harmonic of degree $\ell$.  Let
$\mathbb{H}^{d}_{\ell}$ be the space of all spherical harmonics of
degree $\ell$, and for $L\in\mathbb N_0$, $\Pi_L^{d}$ the space of
all spherical polynomials of degree up to $L$. Then,
$\Pi_L^d=\bigoplus_{\ell=0}^{L}\mathbb H_{\ell}^d$. The dimension of
$\mathbb{H}^{d}_{\ell}$ is
\begin{equation*}
     Z_{d,\ell}:=\mbox{dim}\ \mathbb{H}^{d}_{\ell}=\left\{\begin{array}{ll}
                     \displaystyle\frac{2\ell+d-1}{\ell+d-1}{{\ell+d-1}\choose{\ell}}, & \ell\geq 1,\\[2mm]
                      1, & \ell=0,
                     \end{array}
                       \right.
\end{equation*}
and then the dimension of $\Pi_L^{d}$ is $\sum_{\ell=0}^L
Z_{d,\ell}=Z_{d+1,L} {\asymp} L^d$.  {Here for two sequences
$a_{\ell}$ and $b_{\ell}$, $\ell\in\mathbb N_0$, $a_{\ell}\asymp
b_{\ell}$ means that there exists constants $c,c'$ such that $c'
b_{\ell} \leq a_{\ell} \leq c b_{\ell}$ for all $\ell\in\mathbb
N_0$.} {The Laplace-Beltrami operator $\Delta^*$ on $\mathbb S^d$
has the eigenfunctions $Y_{\ell,m}$ and eigenvalues
$\lambda_{\ell}:=\ell(\ell+d-1)$, $\ell\in \mathbb N_0$,
$m=1,\dots,Z_{d,\ell}$:
\begin{equation*}
    \Delta^*Y_{\ell,m} = \lambda_\ell Y_{\ell,m}.
\end{equation*}}

Let $P_{\ell}^{(d+1)}(t)$, $t\in[-1,1]$, be the normalized
 Gegenbauer polynomial which satisfies
$ P_{\ell}^{(d+1)}(1)=1$ and  the orthogonality relation
$$
                      \int_{-1}^1P_{\ell}^{(d+1)}(t)P_{\ell'}^{(d+1)}(t)(1-t^2)^{\frac{d-2}2} \mathrm{d}t
                      =\frac{|\mathbb S^d|}{|\mathbb
                      S^{d-1}|Z_{d,\ell}}\delta_{\ell,\ell'},
$$
where $\delta_{\ell,\ell'}$ is the Kronecker symbol. Let
$\eta:[0,\infty)\rightarrow\mathbb R$ be a filter with specified
smoothness $\kappa\geq1$ satisfying
\begin{equation}\label{Filter}
              \eta\in C^{\kappa}(\mathbb R_+);\quad
               \mbox{supp} \:\eta\subseteq[1/2,2];\quad
               {\eta(t)^2+\eta(2t)^2}=1\ \mbox{for} \ t\in[1/2,1].
\end{equation}
The filtered kernel $K_n(\mathbf x\cdot \mathbf x')$, $n\geq1$, is
then given by
\begin{equation}\label{filtered kernel}
                     K_n(\mathbf x\cdot \mathbf x')=\sum_{\ell=0}^\infty\eta\left(\frac{\ell}{n}\right)
                     \frac{Z_{d,\ell}}{|\mathbb S^d|}P_{\ell}^{(d+1)}(\mathbf x\cdot
                     \mathbf x'),
\end{equation}
see \cite{NaPeWa2006}. The approximation property of the filtered kernel depends on the smoothness of the filter $\eta$. We refer readers to e.g. \cite[P.101]{SlWo2012} for examples of filtered kernels satisfying \eqref{Filter}. Since the support of $\eta$ is in $[1/2,2]$, $K_n(\mathbf x\cdot\mathbf y)$ is a polynomial of either $\mathbf x$ or $\mathbf y$ of degree up to $2n-1$.

\subsection{Spherical quadrature rules}
The geometric properties of a finite set $X_N:=\{\mathbf x_1,\dots,\mathbf x_N\}$, $N\geq2$, of points on $\mathbb S^d$
can be described by mesh norm, separation radius and mesh ratio, as
we introduce now. The \emph{mesh norm} (or covering radius) of $X_N$
is
\begin{equation*}
	h(X_N):=\max_{\mathbf x\in\mathbb S^d}\min_{\mathbf x_i\in X_N}\mbox{dist}(\mathbf x,\mathbf x_i).
\end{equation*}
The mesh norm is the minimal radius with which the  caps with
centers at points of $X_N$ covers $\mathbb S^d$. The separation
radius of $X_N$ is
\begin{equation*}
     \delta(X_N):=\frac12\min_{j\neq k}\mbox{dist}(\mathbf x_j,\mathbf x_k).
\end{equation*}
This is half the smallest geodesic distance between any pair of
 points in $X_N$. The mesh ratio of $X_N$ is the minimum of distances between points of $X_N$:
\begin{equation*}
    \rho(X_N):=\frac{h(X_N)}{\delta(X_N)}\geq1,
\end{equation*}
which measures how uniformly the points of $X_N$ are distributed on
$\mathbb S^d$. We say a sequence of point sets
$\{X_N\}_{N=2}^{\infty}$ \emph{$\tau$-quasi uniform}, if there is a
constant $\tau\geq 2$ such that $\rho(X_N)\leq \tau$ for all
$N\geq2$. The existence of a $\tau$-quasi uniform sequence of points
sets is proved in \cite{NaSuWaWe2007}. Assume the sequence of points
sets $\{X_N\}_{N=2}^{\infty}$ is $\tau$-quasi uniform, then
 {\begin{equation}\label{meshnorm}
    {h(X_N)}\leq \tau \delta(X_N)\leq\frac{\tau}{2N^{1/d}}.
\end{equation}}
It then follows from \eqref{meshnorm} and \cite[Lemma
2]{LiCaChXu2012} that
\begin{equation*}
    \mathbb S^{d}\subseteq\bigcup_{\mathbf x_i\in X_N}\mathcal C(\mathbf x_i,\tau/(2N^{1/d}))\ \
\mbox{and}\ \ \max_{\mathbf x_i\in X_N}\left|X_N\bigcap \mathcal
C(\mathbf{x}_i,\tau/(2N^{1/d}))\right|\leq 2\pi^{d-1}\tau^d,
\end{equation*}
where $|A|$  {is} the cardinality of the finite set $A$ and  {
 $\mathcal C(\mathbf x,r):=\{\mathbf y\in\mathbf S^{d}: \mbox{dist}(\mathbf x,\mathbf y)\leq r\}$ the spherical cap with center $\mathbf x$
and radius $r$, $r>0$.}

 We say a set $\mathcal Q_N:=\{(w_i,\mathbf x_i): w_i\in \mathbb R
\hbox{~and~} \mathbf x_i\in\mathbb S^d, i=1,\dots,N\}$, $N\geq2$, a
\emph{quadrature rule} on $\mathbb S^d$, where $w_i$ are called weights of $\mathcal Q_N$. For $L\geq0$, a quadrature rule $\mathcal Q_N$ is said to be exact for polynomials of degree up to $L$ if
\begin{equation}\label{eq:quadrature}
    \int_{\mathbb S^d}P(\mathbf x)\mathrm{d}\omega(\mathbf x)=\sum_{i=1}^{N}w_i P(\mathbf x_i) \quad \forall P\in \Pi_L^d.
\end{equation}

The following lemma gives a sequence of polynomial-exact quadrature rules whose point sets are $\tau$-quasi uniform, see \cite[Theorem 3.1]{BrDa2005} and
also \cite{MhNaWa2001}.
\begin{lemma}[\cite{BrDa2005,MhNaWa2001}]\label{Lemma:fixed cubature}
If  $\{X_N\}_{N=2}^{\infty}$ is $\tau$-quasi uniform, then for
$N\geq2$, there exist positive weights $w_i$, $i=1,\dots,N$, such
that $0\leq w_i\leq c_2N^{-1}$ and
$$
    \int_{\mathbb S^d}P(\mathbf x)\mathrm{d}\omega(\mathbf x)
    = \sum_{\mathbf x_i\in X_N}w_iP(\mathbf x_i)\quad\forall P\in\Pi_{c_3N^{1/d}}^{d},
$$
where $c_2$ and $c_3$ are constants depending only on $\tau$ and
$d$.
\end{lemma}

For $1\leq p<\infty$, let $L_{p,\mu}:=L_{p}(\mathbb S^{d},\mu)$ be the $L_p$ space on $\mathbb S^d$ with respect to a probability measure
$\mu$, endowed with norm $\|f\|_{p,\mu}:=\left(\int_{\mathbb S^d}
|f(\mathbf x)|^p \mathrm{d}\mu(\mathbf x)\right)^{1/p}$. The following
theorem shows that if $X_N=\{\mathbf x_i\}_{i=1}^N\subset\mathbb
S^{d}$ is a set of i.i.d. random points with distribution
$\mu$, then with high probability, the quadrature rule is exact
for polynomials with a specific degree.

\begin{theorem}\label{Theorem:random cubature}
 Let $X_N=\{\mathbf x_i\}_{i=1}^{N}$ be i.i.d. random points on $\sph{d}$ with distribution $\mu$ which satisfies
\begin{equation}\label{condition on distribution}
         \|f\|_{L_1(\mathbb S^d)} \leq c_4 \|f\|_{1,\mu}\quad\forall
         f\in L_1(\mathbb S^d)\cap L_{1,\mu}
\end{equation}
 for a positive absolute constant $c_4$. Then, for integer $N$ satisfying $N/n^{2d}>c$ for a sufficiently large constant $c$,
  there exits a quadrature rule $\mathcal Q_N:=\{(\mathbf x_i,w_{i,n})\}_{i=1}^{N}$ such that
\begin{equation*}
	\int_{\mathbb S^d}P_{n}(\mathbf x)\mathrm{d}\mu(\mathbf x)=
    \sum_{i=1}^{N}w_{i,n}P_{n}(\mathbf x_i) \quad\forall P_n\in
    \Pi_n^d
\end{equation*}
holds, and $\sum_{i=1}^{N}| {w_{i,n}}|^2\leq\frac{2}{{N}}$ and $w_{i,n}\geq0$ for all $i=1,\dots,N$, with confidence at least $1-4\exp\left\{-CN/n^d\right\}$,
where $C$ is a constant depending only on $c_4$ and $d$. 
\end{theorem}

We give the proof of Theorem~\ref{Theorem:random cubature} in
Section~\ref{Sec.Proof}. The confidence level in Theorem~\ref{Theorem:random cubature} is exponential of $n$ and $N$, which is different from polynomial confidence level given by \cite[Theorem~4.1]{LeMh2008}. This exponential relation plays a crucial role in error estimation for the distributed filtered hyperinterpolation for spherical noisy data. It is also different from \cite{Lin2018} in that the established quadrature rule satisfying \eqref{condition on distribution} holds with positive weights. The condition of the distribution in \eqref{condition on distribution} describes the distortion of $\mu$ from the uniform distribution on $\sph{d}$ that is also the spherical Lebesgue measure. The probabilistic quadrature rule in Theorem~\ref{Theorem:random cubature} is critical to the construction of distributed filtered hyperinterpolation with random sampling (see Section~\ref{sec:distrfih.ransamp}).

\subsection{Spherical filtered hyperinterpolation approximation}
Using filtered kernel and quadrature rule, we define the
spherical filtered hyperinterpolation, as follows.
\begin{definition}[\cite{SlWo2012}] Let $d\geq2$ and $1\leq p\leq\infty$. 
    For $f\in L_p(\mathbb S^d)$, the filtered hyperinterpolation
(approximation) with a quadrature rule $\mathcal Q_N:=\{(w_i,\mathbf
x_i)\}_{i=1}^{N}$ is
\begin{equation}\label{filtered hi}
       V_{n,N}(f;\mathbf x):=\sum_{i=1}^N w_{i}f(\mathbf x_i)K_n(\mathbf x_i\cdot \mathbf x),\quad n\geq1.
\end{equation}
\end{definition}

The approximation property of the filtered hyperinterpolation
depends on the smoothness of the function space.  For $r\in\mathbb
R_+$, let {$$
       b_{\ell}^r:=(1+\lambda_\ell)^{r/2}\asymp (1+\ell)^r.
$$}
Let $\{Y_{\ell,m}:\ell=0,1,\dots,\: m=1,\dots,Z_{d,\ell}\}$ be an
orthonormal basis for the space $L_2(\mathbb S^d)$, and 
$$
    \widehat{f}_{\ell m}:=\langle f,Y_{\ell,m}\rangle_{L_2(\mathbb S^d)}
    :=\int_{\mathbb S^d}f(\mathbf x)Y_{\ell m}(\mathbf
    x)\mathrm{d}\omega(\mathbf x)
$$
the Fourier coefficients of $f\in L_2(\mathbb S^d)$. For
$d\geq2$, $1\leq p\leq \infty$ and $r\in\mathbb R_+$, the Sobolev
space $\mathbb W_p^r(\mathbb S^d)$ is the space of functions $f$
satisfying $\sum_{\ell=0}^{\infty}b_{\ell}^r \widehat{f}_{\ell
m}Y_{\ell,m}\in L_p(\mathbb S^d)$, endowed with the norm
$\|f\|_{\mathbb W^r_p(\mathbb
S^d)}:=\left\|\sum_{\ell=0}^{\infty}b_{\ell}^r \widehat{f}_{\ell m}
Y_{\ell,m}\right\|_{L_p(\mathbb S^d)}$.

 {The following lemma, which is proved by Wang and Sloan
\cite{WaSl2017}, shows the relation of the approximation error of
the filtered hyperinterpolation approximation $V_{n,N}$ and the
smoothness of the function space under the condition that the
associated quadrature rule is exact for polynomials of degree up to
$3n-1$.}

\begin{lemma}[\cite{WaSl2017}]\label{Lemma:Filtered h app}
Let $d\geq 2$, $1\leq p\leq \infty$ and $r>d/p$. Let $V_{n,N}$ be
the filtered hyperinterpolation in \eqref{filtered hi} with
quadrature rule $\mathcal Q_N$ exact for polynomials of
degree up to $3n-1$ and with the filter $\eta$ in $C^\kappa(\mathbb R_+)$, $\kappa\geq\left\lfloor\frac{d+3}2\right\rfloor$. Then, for $f\in \mathbb W_p^r(\mathbb S^d)$,
\begin{equation}\label{eq:fihapp.ub}
    \|f-V_{n,N}(f)\|_{L_p(\mathbb S^d)}\leq c_5\: n^{-r}\|f\|_{\mathbb W_p^r(\mathbb S^d)},
\end{equation}
where $c_5$ depends only upon $d,p,r$ and $\eta$.  {The order
$n^{-r}$ in \eqref{eq:fihapp.ub} is optimal.}
\end{lemma}

\section{Distributed Filtered Hyperinterpolation with Deterministic
Sampling}\label{Sec.Fixed design}

 {A \emph{data set} $D=\{(\mathbf x_i,y_i)\}_{i=1}^{|D|}$ on
$\mathbb S^d$ is a set of pairs of points $\Lambda_{D}:=\{\mathbf
x_i\}_{i=1}^{|D|}$ on the sphere and real numbers $y_i$. Elements of
$D$ are called \emph{data}. The points $\mathbf x_i$ of
$\Lambda_{D}$ are called the \emph{sampling points} of $D$. To distinguish quadrature rule with random points which we will investigate later, we say a data set $D$ has \emph{deterministic sampling} for (fixed) sampling points.
In this section, we introduce a new filtered hyperinterpolation
for which distributed learning can be used. The data $y_i$ are
the values of a function $f^*$ on $\mathbb S^d$ plus noise, that
is,
\begin{equation}\label{eq:noisydats}
        y_i=f^*(\mathbf x_i)+\epsilon_i, \quad \mathbf E[\epsilon_i]=0,\quad |\epsilon_i|\leq M \quad\forall
        i=1,\dots,|D|.
\end{equation}
The $D$ satisfying \eqref{eq:noisydats} is then called
\emph{noisy data set associated with $f^*$}.

\subsection{Filtered hyperinterpolation for noisy data: deterministic sampling}
We first study the performance of the filtered
hyperinterpolation for a noisy data set $D$ whose data are stored on
a ``big enough'' machine.

\begin{definition}\label{defn:fihdetsamp} For $s\in\mathbb N_0$ and $\{\mathbf x_i\}_{i=1}^{|D|}$, let $\mathcal
Q_{|D|}$ be the quadrature rule given by Lemma~\ref{Lemma:fixed cubature}, which is exact for polynomials of degree up to $s$ and has positive weights $\{w_{i,s,D}\}_{i=1}^{|D|}$ satisfying $0\leq w_{i,s,D}\leq c_2|D|^{-1}$.
    The filtered hyperinterpolation for noisy data associated with a function $f^*$ on $\mathbb S^d$ is 
\begin{equation}\label{SFH for noisy data1}
     f^\diamond_{D,n}(x):= \sum_{i=1}^{|D|}w_{i,s,D}\:y_i K_n(\mathbf x_i\cdot\mathbf x),
\end{equation}
where $K_n$ is a filtered kernel in \eqref{filtered kernel} for
$\eta\in C^\kappa(\mathbb R_+)$ with
$\kappa\geq\left\lfloor\frac{d+3}2\right\rfloor$ and $n\leq s$.
\end{definition}

The kernel $K_n$ provides a smoothing method for the function $f^*$ using data $D$. As we shall see below, the approximation error of this filtered hyperterpolation has the convergence rate depending on the smoothness of function $f^*$.

If $\Lambda_{D}$ is $\tau$-quasi uniform, it then follows from Lemma
\ref{Lemma:fixed cubature} that $s=c_3N^{1/d}$.  We do not
assume the magnitude of the noise extremely small, the filter
$\eta$ of the filtered kernel $K_n$ shall then be chosen properly to minimize the impact of noise on interpolation. This is a problem similar to ``model selection'' in the statistical and machine
learning \cite{CuZh2007}. To say it precisely, if the support $n$ is
too large, then the filtered hyperinterpolation $f^\diamond_{D,n}$
will have precise approximation at the data set $\{(\mathbf
x_i,y_i)\}_{i=1}^{D}$, but $f^\diamond_{D,n}$ may not be a good
approximation of $f^*$ due to the noise. If the support is too small, the performance of the filtered hyperinterpolation $f^\diamond_{D,n}(\mathbf x_i)$ is not good, even at the interpolation points. It is then preferable to set
$n$ as a parameter in the training process. For $\Lambda_{D}$, the quadrature rule in Definition~\ref{defn:fihdetsamp}
(which is from Lemma~\ref{Lemma:fixed cubature}) is valid for
$n$ sufficiently large. The parameter selection thus needs
only a few steps of computation, while the filtered hyperterpolation in Definition~\ref{defn:fihdetsamp} allows us to handle massive noisy spherical data. This property of filtered hyperinterpolation is different from other algorithms such as the regularized least squares \cite{HeSlWo2017}: the latter needs to compute the inverse of the kernel matrix for each regularization parameter.

The following theorem shows that the filtered hyperinterpolation
$f^\diamond_{D,n}$ can approximate $f^*$ well, provided that the
support of the filtered kernel is appropriately tuned and the
sampling point set $\Lambda_{D}$ is $\tau$-quasi uniform for $\tau\geq 2$.

\begin{theorem}\label{Theorem:learning rate fixed design}
Let $d\geq 2$ and $r>d/2$. Assume the sampling point set $\Lambda_{D}$ of the data set $D$ is $\tau$-quasi uniform for $\tau\geq 2$, and $\frac{c_3}{6}|D|^{1/(2r+d)}\leq n\leq \frac{c_3}3 |D|^{1/(2r+d)}$ for constant $c_3$ in Lemma~\ref{Lemma:fixed cubature}. Then, the filtered
hyperinterpolation $f_{D,n}^{\diamond}$ for noisy data set $D$ with
target function $f^*\in \mathbb W_2^r(\mathbb S^d)$ satisfies
\begin{equation}\label{error 1}
     \mathbf E\left\{\|f^\diamond_{D,n}-f^*\|_{L_2(\mathbb S^d)}^2\right\}
     \leq C_1 |D|^{-2r/(2r+d)},
\end{equation}
where $C_1$ is a constant independent of $|D|$ and $n$.
\end{theorem}

 {\begin{remark}
    Here the condition $r>d/2$ is the embedding
     condition such that any function in $\mathbb W_2^r(\mathbb S^d)$ has
      a representation of a continuous function on $\mathbb S^d$.
  The numerical computation of the filtered hyperinterpolation then makes sense.
\end{remark}}

{We give the proof of Theorem~\ref{Theorem:learning rate fixed
design} in Section~\ref{Sec.Proof}.} As mentioned above, since
$\Lambda_{D}$ is $\tau$-quasi uniform, the quadrature rule for the
filtered hyperinterpolation in Definition~\ref{defn:fihdetsamp} is
exact for $\Pi_s^d$ with $s=c|D|^{1/d}$. As we choose $n\leq c
|D|^{1/(2r+d)}\leq s$ in Theorem~\ref{Theorem:learning rate fixed
design}, $f^\diamond_{D,n}$ reproduces polynomials in $\Pi_s^d$.
Theorem~\ref{Theorem:learning rate fixed design} illustrates that if
the scattered data $\Lambda_{D}$ has good geometric property, for
example, $\tau$-quasi uniformity, and the support of the filter $\eta$ is appropriately chosen, then the spherical filtered hyperinterpolation for noisy data set
$D$ can approximate a sufficiently smooth target function on the sphere in high precision in probablistic sense. By
\cite{GyKoKrWa2002}, the rate $|D|^{-2r/(2r+d)}$ in \eqref{error 1}
cannot be essentially improved in the scenario of
\eqref{eq:noisydats}. Thus, Theorem~\ref{Theorem:learning rate fixed
design} provides a feasibility analysis of the spherical filtered
hyperinterpolation for spherical data with random noise.

\subsection{Distributed filtered hyperinterpolation: deterministic sampling}
We say a large data set $D$ is \emph{distributively stored} on $m$
local machines if for $j=1,\dots,m$, $m\geq2$, the $j$th machine contains a
subset $D_j$ of $D$, and there is no common data between any pair of
machines, that is, $D_j\cap D_{j'}=\emptyset$ for $j\neq j'$, and
$D=\cup_{j=1}^m D_j$. The data sets $D_1,\dots,D_m$ are called
\emph{distributed data sets} of $D$. In this case, the filtered
hyperinterpolation $f^\diamond_{D,n}$ which needs access to the
entire data set $D$ is infeasible. Instead, in this section, we
construct a \emph{distributed filtered hyperinterpolation for the
distributed data sets $\{D_j\}_{j=1}^{m}$ of $D$} by the divide and
conquer strategy \cite{LiGuZh2017}.

\begin{definition}\label{defn:distrfih}
The distributed filtered hyperinterpolation $\overline{f^\diamond_{D,n}}$ for distributed data sets
$\{D_j\}_{j=1}^{m}$ of a noisy data set $D$ associated with function
$f^*$ on $\mathbb S^d$ is a synthesized estimator of local
estimators $f^\diamond_{D_j,n}$, $j=1,2,\dots,m$, each of which is
the spherical filtered hyperinterpolation \eqref{SFH for noisy
data1} for noisy data set $D_j$:
\begin{equation}\label{distrilearn 1}
   \overline{f^\diamond}(\{D_j\}_{j=1}^{m},n;\mathbf x) :=
   \overline{f^\diamond_{D,n}}(\mathbf x) := \sum_{j=1}^m \frac{|D_j|}{|D|}
   f^\diamond_{D_j,n}(\mathbf x),\quad \bx\in\mathbb S^{d}.
\end{equation}
\end{definition}
The synthesis here is a process when the local estimators
communicate to a central processor to produce the global estimator
$\overline{f^\diamond_{D,n}}$.

 {The following theorem shows that the distributed filtered
hyperinterpolation} $\overline{f^\diamond_{D,n}}$ has similar
approximation performance as $f^\diamond_{D,n}$  {if the number of
local machines is not large}.

\begin{theorem}\label{Theorem:distributed learning rate fixed design}
Let $d\geq 2$, $r>d/2$, $m\geq2$ and $D$ a noisy data set
satisfying \eqref{eq:noisydats}. Let $\{D_j\}_{j=1}^{m}$ be $m$
distributed data sets of $D$. For $j=1,\dots,m$, the sampling
point set $\Lambda_{D_j}$ of $D$ is $\tau$-quasi uniform for
$\tau\geq 2$.  {If the distributed filtered hyperinterpolation
$\overline{f^\diamond_{D,n}}$ for $\{D_j\}_{j=1}^{m}$ satisfies that
the target function $f^*$ is in $\mathbb W_2^r(\mathbb S^d)$, $\frac{c_3}{6}|D|^{1/(2r+d)}\leq n\leq \frac{c_3}3 |D|^{1/(2r+d)}$ for constant $c_3$ in Lemma~\ref{Lemma:fixed cubature}, and
$\min_{j=1,\dots,m}|D_j|\geq |D|^{\frac{d}{2r+d}}$, then,}
\begin{equation}\label{error 1.1}
            \mathbf E\left\{\|\overline{f^\diamond_{D,n}}-f^*\|_{L_2(\mathbb S^d)}^2\right\}\leq
            C_2 |D|^{-2r/(2r+d)},
\end{equation}
where $C_2$ is a constant independent of $|D|$, $|D_1|,\dots,|D_m|$
and $n$.
\end{theorem}

{\begin{remark} The condition $\min_{j=1,\dots,m}|D_j|\geq
|D|^{\frac{d}{2r+d}}$ has a close connection to the number $m$ of
local machines. In particular, if $|D_1|=\dots=|D_m|$, since each
$D_j$ is $\tau$-quasi uniform, $\min_{j=1,\dots,m}|D_j|\geq
|D|^{\frac{d}{2r+d}}$ is equivalent to $m\leq |D|^\frac{2r}{2r+d}$.
\end{remark}}

The proof of Theorem~\ref{Theorem:distributed learning rate fixed
design} is postponed in Section~\ref{Sec.Proof}.
Theorem~\ref{Theorem:distributed learning rate fixed design}
illustrates that if $\min_{j=1,\dots,m}|D_j|\geq
|D|^{\frac{d}{2r+d}}$, then with the same assumption as Theorem
\ref{Theorem:learning rate fixed design}, the distributed filtered
hyperinterpolation will have the same approximation performance
as the filtered hyperinterpolation that treats all the distributed data sets as a whole ``big enough'' machine.

% Distributed filtered hyperinterpolation with random sampling
\section{Distributed Filtered Hyperinterpolation with Random
Sampling}\label{Sec.Random}  We say a data set $D$ has \emph{random
sampling} if its sampling points are i.i.d. random points on $\sph{d}$. In this section, we construct a filtered hyperinterpolation for noisy data satisfying \eqref{eq:noisydats} with random sampling points.

\subsection{Filtered hyperinterpolation for noisy data: random sampling}
The filtered hyperinterpolation for noisy data with random
sampling can be constructed as follows.
Let $D=\{(\mathbf x_i,y_i)\}_{i=1}^{|D|}$ and $n\in\mathbb N$. Let
the quadrature rule $\mathcal Q_{|D|}:=\{(a^*_{i,n,D},\mathbf
x_i)\}_{i=1}^{|D|}$ as given by
Theorem~\ref{Theorem:random cubature}, which is exact for polynomials of degree $n$. For $m\geq2$, let
\begin{equation}\label{eq:fihransamp}
    a_{i,n,D}=\left\{\begin{array}{ll} a^*_{i,n,D},& \displaystyle\mbox{if}\
        \sum_{i=1}^{|D|}|a^*_{i,n,D}|^2\leq 2/m,\\[1mm]
        0,&\mbox{otherwise},
        \end{array}\right.\quad\forall i=1,\dots,|D|.
\end{equation}
\begin{definition}\label{defn:fihransamp}
The filtered hyperinterpolation for noisy data $D:=\{(\mathbf
x_i,y_i)\}_{i=1}^{|D|}$ with random sampling points $\{\mathbf
x_i\}_{i=1}^{|D|}$ is
\begin{equation}\label{SFH for noisy data}
     f_{D,n}(\mathbf x):= \sum_{i=1}^{|D|}a_{i,n,D}\:y_i
     K_n(\mathbf x_i\cdot\mathbf x).
\end{equation}
\end{definition}

The following theorem gives the approximation error of the filtered hyperinterpolation in Definition~\ref{defn:fihransamp} for sufficiently smooth functions.
\begin{theorem}\label{Theorem:learning rate fih random samp}
Let $d\geq 2$ and $r>d/2$. Let the noisy data set $D$ with i.i.d. random sampling points on $\mathbb S^d$ and 
distribution $\mu$ satisfying \eqref{condition on
distribution}. For integer $n$ satisfying $\frac{c_3}{6}|D|^{1/(2r+d)}\leq n\leq \frac{c_3}3 |D|^{1/(2r+d)}$ with constant $c_3$ in Lemma~\ref{Lemma:fixed cubature},
 the filtered hyperinterpolation $f_{D,n}$ for noisy data set
$D$ with target function $f^*\in \mathbb W_2^r(\mathbb S^d)$ has
the approximation error
\begin{equation}\label{error 2}
            \mathbf E\left\{\|f_{D,n}-f^*\|_{L_2(\mathbb S^d)}^2\right\}\leq C_3|D|^{-2r/(2r+d)},
\end{equation}
where $C_3$ is a constant independent of $|D|$.
\end{theorem}

We give the proof of Theorem~\ref{Theorem:learning rate fih random
samp} in Section~\ref{Sec.Proof}.
Theorems~\ref{Theorem:learning rate fixed design} and
\ref{Theorem:learning rate fih random samp} show that the filtered
hypinterpolation approximations with random sampling and
deterministic sampling achieve the same optimal convergence rate.

\subsection{Distributed filtered hyperinterpolation: random sampling}\label{sec:distrfih.ransamp}

The distributed filtered hyperinterpolation with random sampling is
a weighted average of filtered hyperinterpolation approximations
for data on local machines. Here the weight for a local machine is the proportion of the data used by the machine to all data. Let $f_{D_j,n}$ be the filtered hyperinterpolation for data $D_j$. Similar as \eqref{distrilearn 1},
the global estimator $\overline{f}_{D,n}$ is defined as follows.
\begin{definition}\label{defn:distrfih.ransamp} Let $d\geq 2$ and $D$ a noisy data set satisfying \eqref{eq:noisydats}. The sampling points of $D$ are i.i.d random points on $\mathbb S^d$. For $m\geq2$, let
$\{D_j\}_{j=1}^{m}$ be $m$ distributed data sets of $D$, and for
$j=1,\dots,m$, $f_{D_j,n}$ be the filtered hyperinterpolation for
$D_j$ given by Definition~\ref{defn:fihransamp}. The distributed
filtered hyperinterpolation for distributed data sets
$\{D_j\}_{j=1}^{m}$ of $D$ is
    \begin{equation}\label{distrilearn}
   \overline{f}_{D,n} := \sum_{j=1}^m \frac{|D_j|}{|D|} f_{D_j,
n}.
\end{equation}
The $f^*$ in \eqref{eq:noisydats} is called target function
to $\overline{f}_{D,n}$.
\end{definition}

The following theorem gives an upper bound of approximation error
for distributed filtered hyperinterpolation with random sampling.
\begin{theorem}\label{Theorem:distributed learning rate}
 Let $d\geq 2$, $r>d/2$, $m\geq2$ and $D$ a noisy data set
satisfying \eqref{eq:noisydats}. The sampling points are i.i.d
random points on $\mathbb S^d$ with distribution $\mu$ in \eqref{condition on distribution}. If the target function $f^*\in \mathbb W_2^r(\mathbb S^d)$, $\frac{c_3}{6}|D|^{1/(2r+d)}\leq n\leq \frac{c_3}3 |D|^{1/(2r+d)}$ with constant $c_3$ in Lemma~\ref{Lemma:fixed cubature},
and $\min_{j=1,\dots,m}|D_j|\geq |D|^{\frac{d+\nu}{2r+d}}$ for some $\nu$ in $(0,2r)$, then,
\begin{equation}\label{error 2.1}
            \mathbf E\left\{\|\overline{f}_{D,n}-f^*\|_{L_2(\mathbb S^d)}^2\right\}
            \leq
            C_4|D|^{-2r/(2r+d)},
\end{equation}
where $C_4$ is a constant independent of $|D|$, $|D_1|,\dots,|D_m|$
and $n$.
\end{theorem}

The proof of Theorem~\ref{Theorem:distributed learning rate} will be
given in Section~\ref{Sec.Proof}.
From Theorems~\ref{Theorem:distributed learning rate} and
\ref{Theorem:distributed learning rate fixed design}, we see that
the distributed filtered hyperinterpolation approximations with
random sampling and deterministic sampling can both achieve the
convergence rate of order $|D|^{-2r/(2r+d)}$. To achieve this
approximation order, the condition on the number
of local machines of the random  sampling is stronger than the
deterministic case, since the former requires
$\min_{j=1,\dots,m}|D_j|\geq |D|^{(d+\nu)/(2r+d)}$ for $\nu \in(0,2r)$, while the latter only needs $\min_{j=1,\dots,m}|D_j|\geq |D|^{d/(2r+d)}$. 

Here we only consider error estimates with respect to Lebesgue measure. It would be interesting to consider error estimates for distributed learning with respect to other measures as done in \cite{Zhou2018distributed,Zhou2018deep,Zhou2019}.

% Numerical Examples
\section{Numerical Examples}\label{sec:numer.exam}
\begin{figure}[htbp]
  \centering
  \begin{minipage}{\textwidth}
  \centering
  \begin{minipage}{0.48\textwidth}
  \centering
  \includegraphics[width=\textwidth]{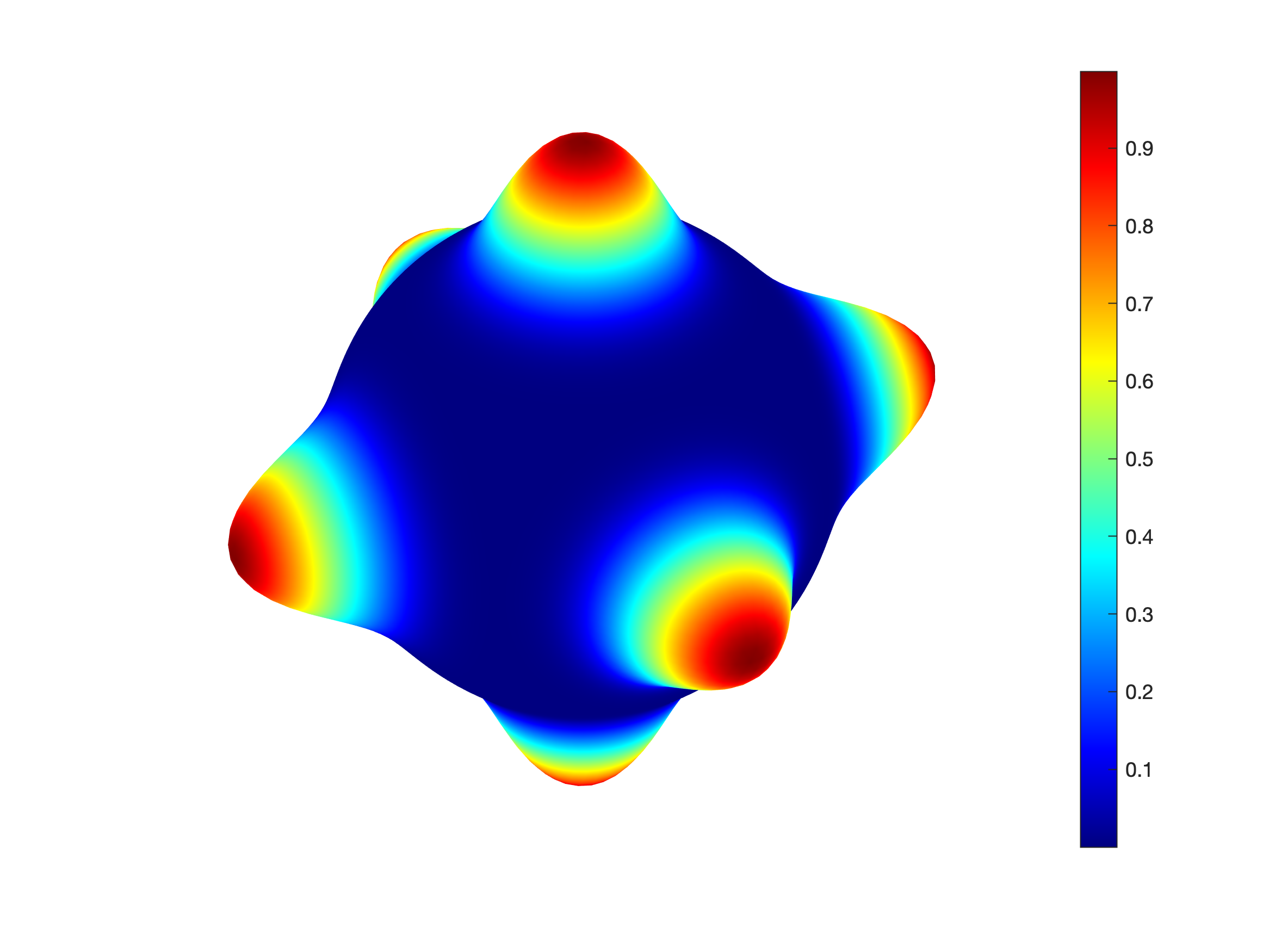}\\[-2mm]
  {\scriptsize (a) $f$}
  \end{minipage}
  \begin{minipage}{0.48\textwidth}
  \centering
  \includegraphics[width=\textwidth]{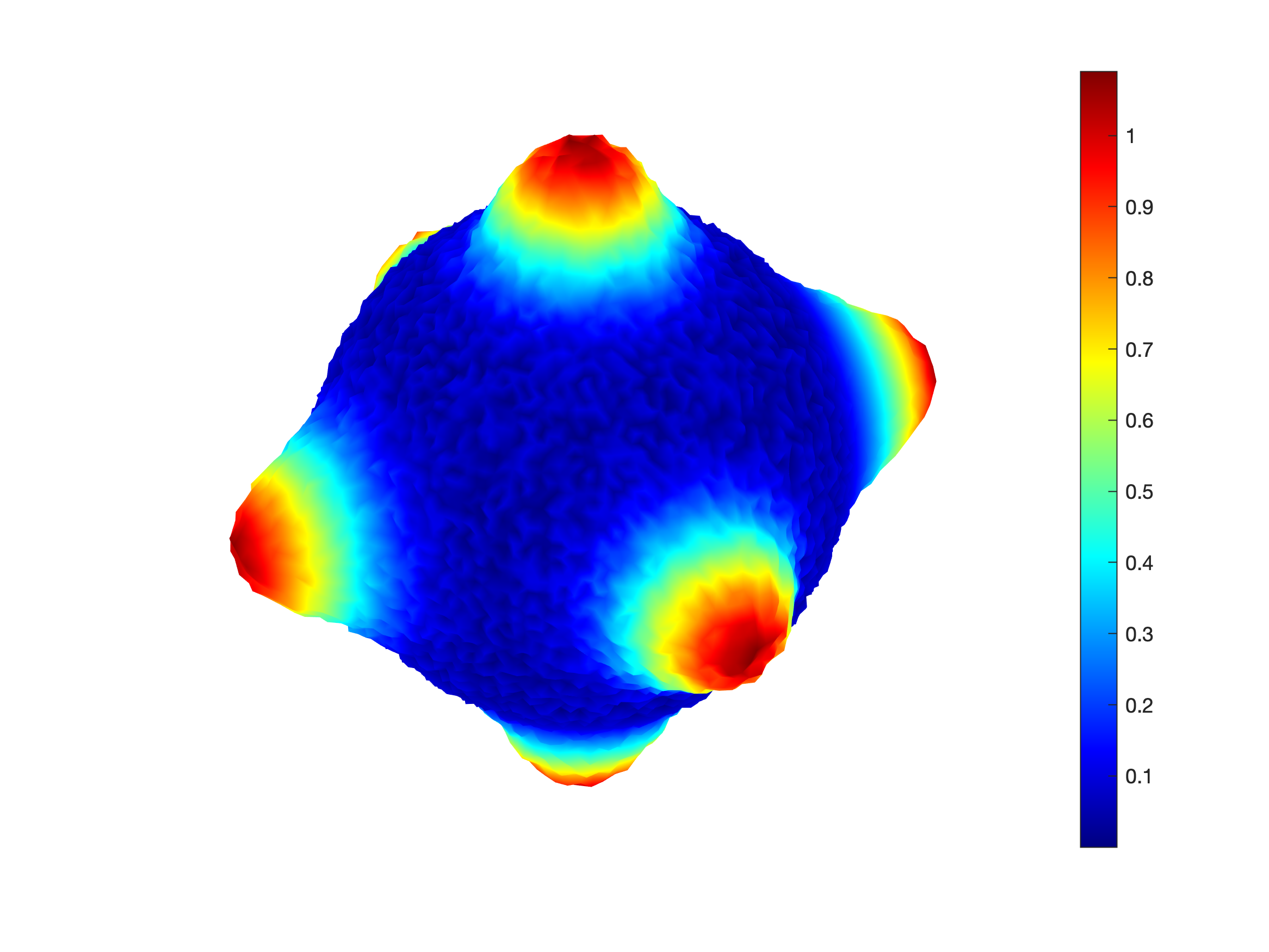}\\[-2mm]
  {\scriptsize (b) $f$ plus noise}
  \end{minipage}
  \end{minipage}\vspace{2mm}
    \begin{minipage}{\textwidth}
  \centering
   \begin{minipage}{0.48\textwidth}
  \centering
  \includegraphics[width=\textwidth]{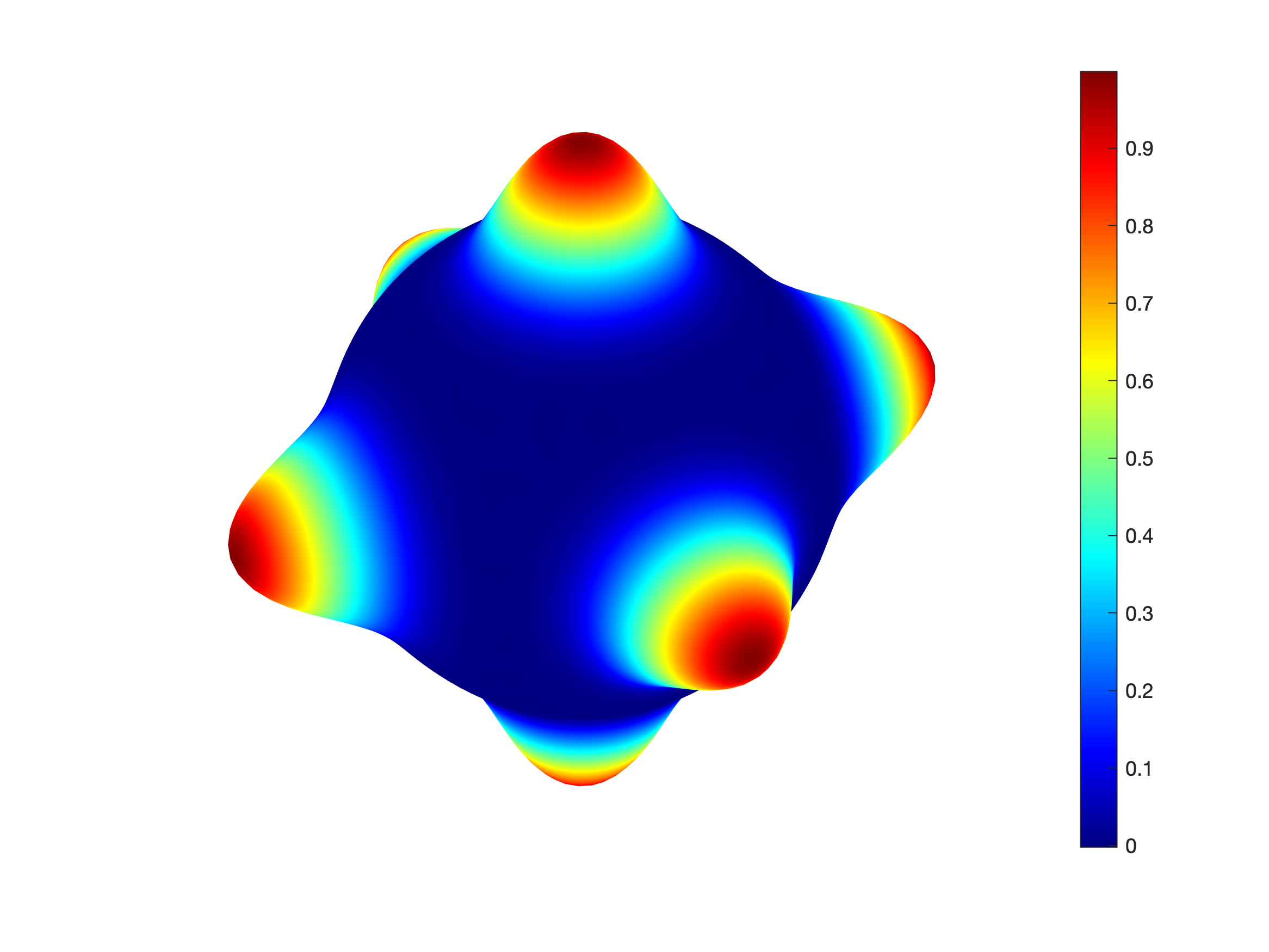}\\[-2mm]
  {\scriptsize (c) $\overline{f^\diamond_{D,n}}$}
  \end{minipage}
  \begin{minipage}{0.48\textwidth}
  \centering
  \includegraphics[width=\textwidth]{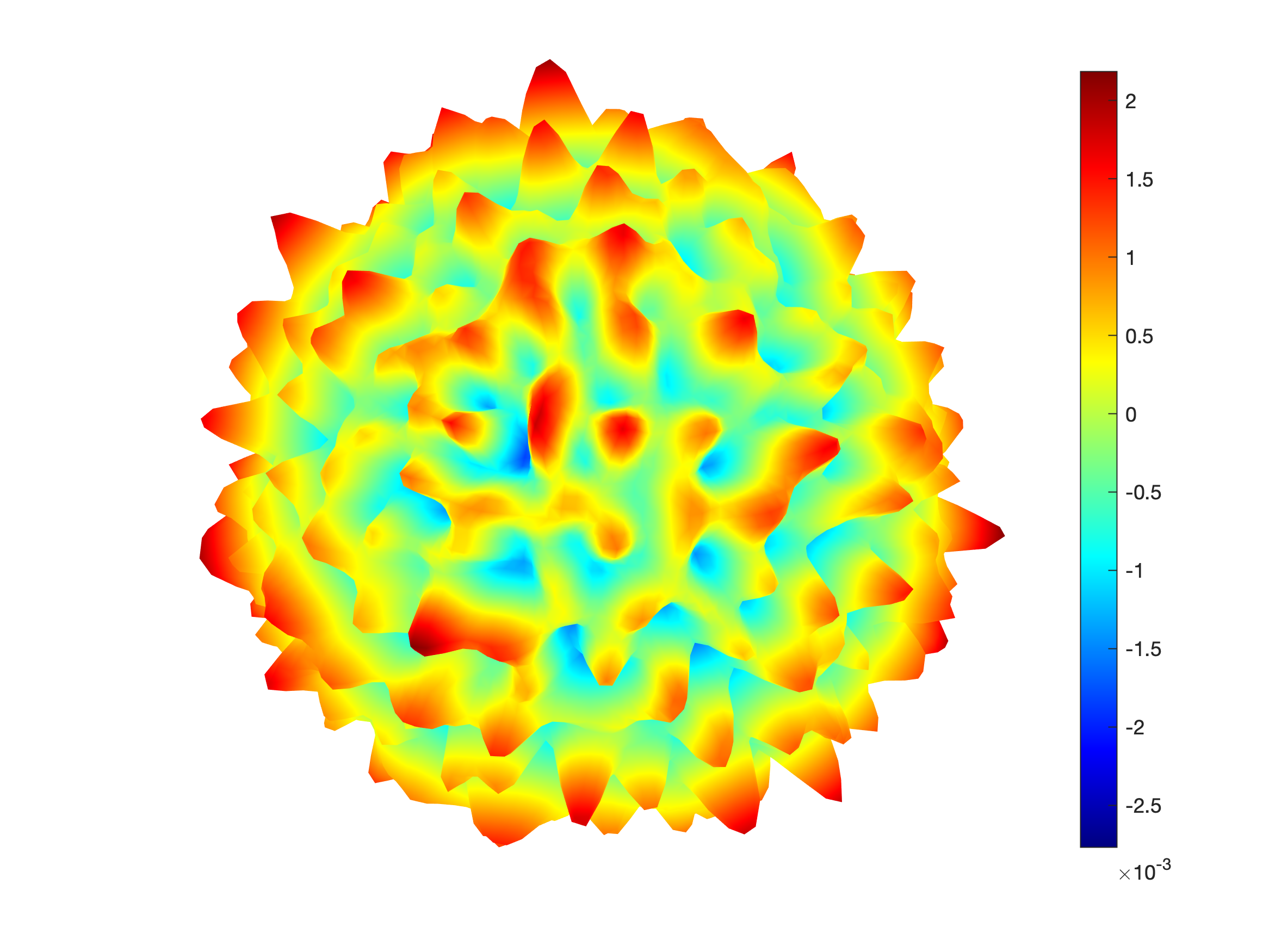}\\[-1mm]
  {\scriptsize (d) Error}
  \end{minipage}
  \end{minipage}
  \caption{(a) Function $f$ in \eqref{eq:rbf} which is in Sobolev space $H^{4.5}(\sph{2})$; (b) Noisy data $f+\epsilon$ with noise level $\sigma=0.1$; (c) Distributed filtered hyperinterpolation $\overline{f^\diamond_{D,n}}$ with $n=25$, $m=100$, $\sigma=0.1$ and the noisy dataset $D$ given on the nodes of symmetric spherical $75$-design; (d) Error $\overline{f^\diamond_{D,n}}-f$.}
  \label{fig:rbf}
\end{figure}

\begin{figure}[ht]
  \centering
  \begin{minipage}{0.7\textwidth}
  \centering
  \includegraphics[width=\textwidth]{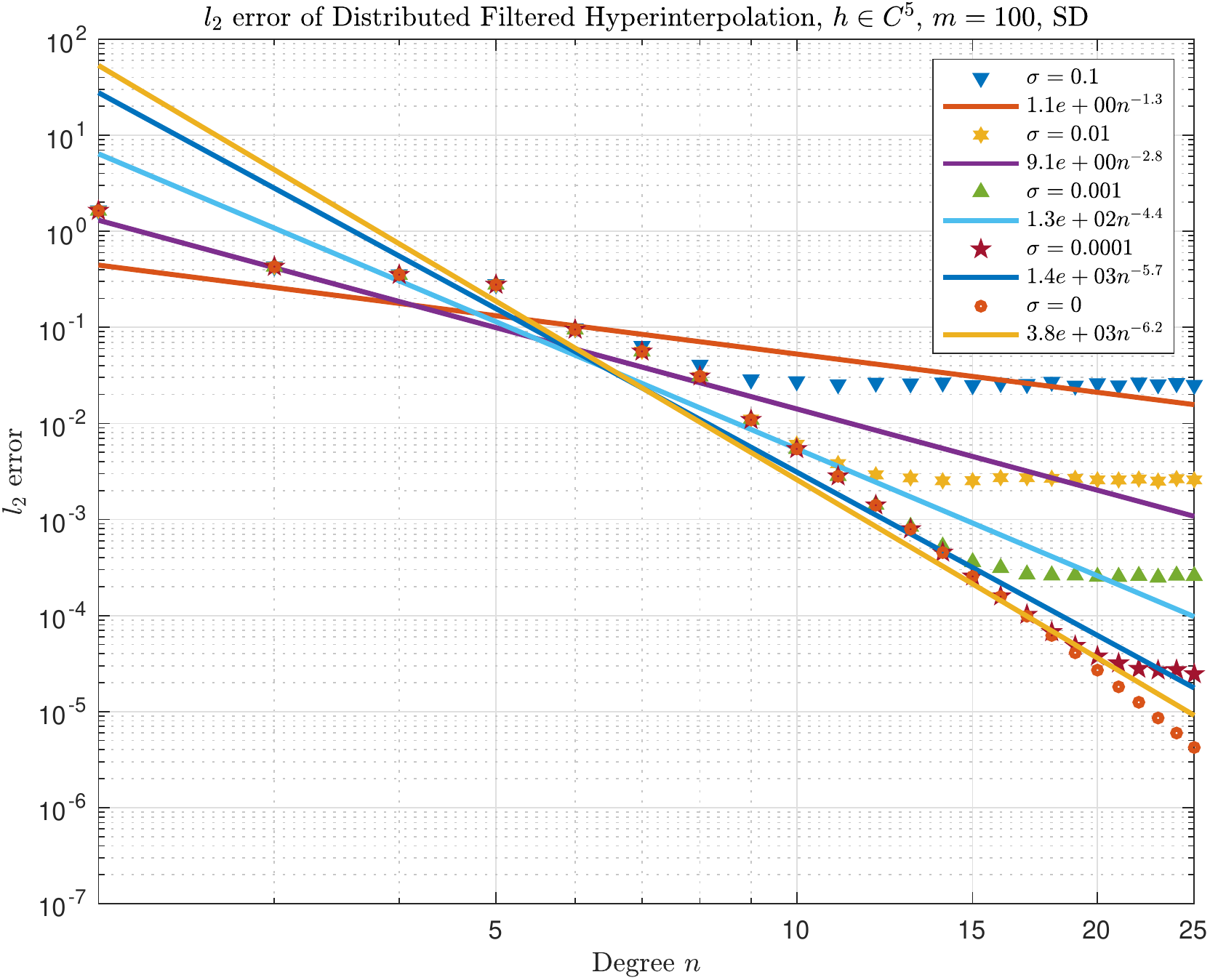}
    \end{minipage}
  \caption{Errors v.s. degree $n$ for distributed filtered hyperinterpolation, $n\leq 25$, $m=100$, $\sigma=0,0.0001,0.001,0.01,0.1$.}
  \label{fig:errvsL}
\end{figure}

\begin{figure}[ht]
  \centering
  \begin{minipage}{0.7\textwidth}
  \centering
  \includegraphics[width=\textwidth]{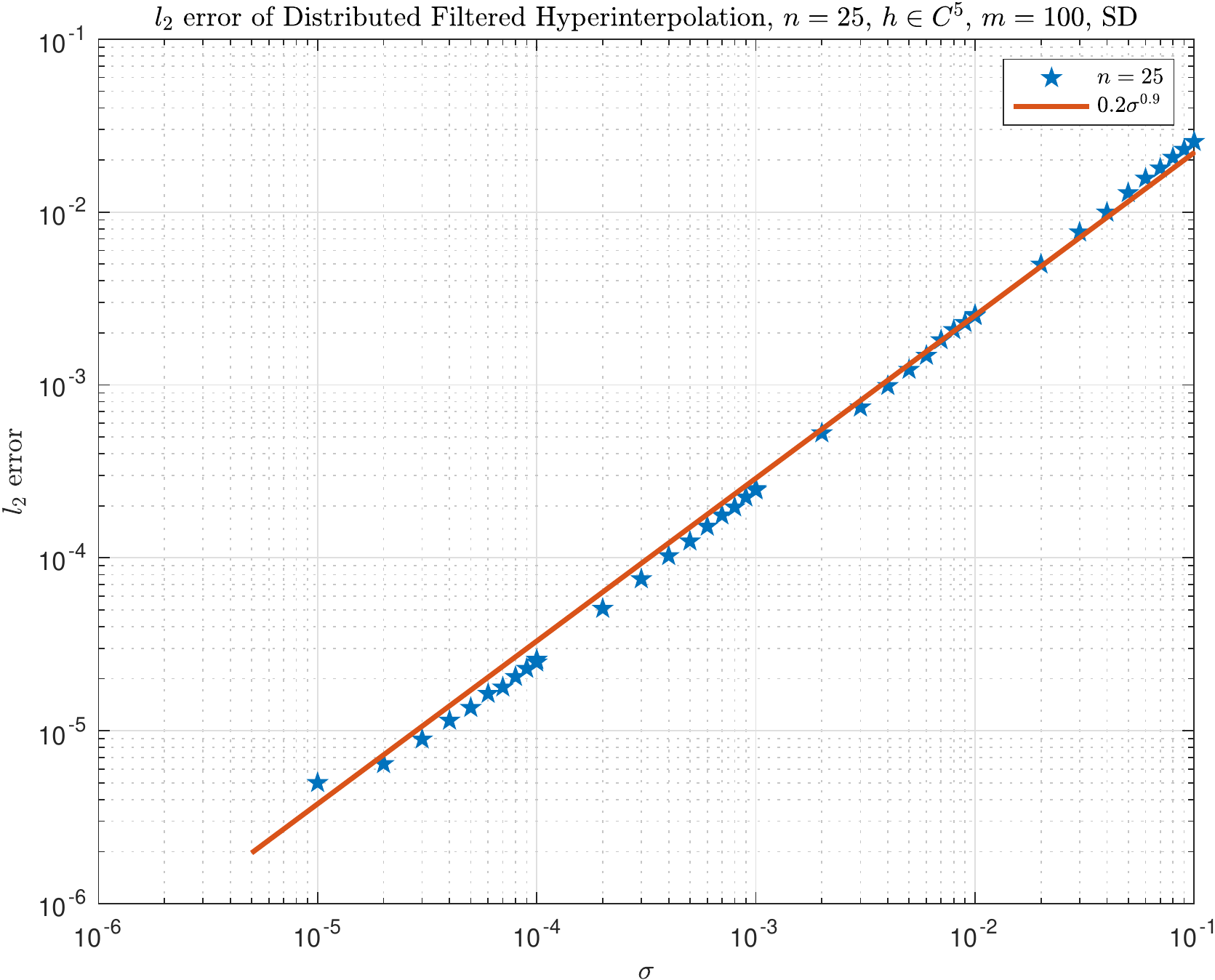}
    \end{minipage}
  \caption{Errors v.s. standard variance $\sigma$ for distributed filtered hyperinterpolation, $n=25$, $m=100$}
  \label{fig:errvssigma}
\end{figure}

\begin{figure}[ht]
  \centering
  \begin{minipage}{0.7\textwidth}
  \centering
  \includegraphics[width=\textwidth]{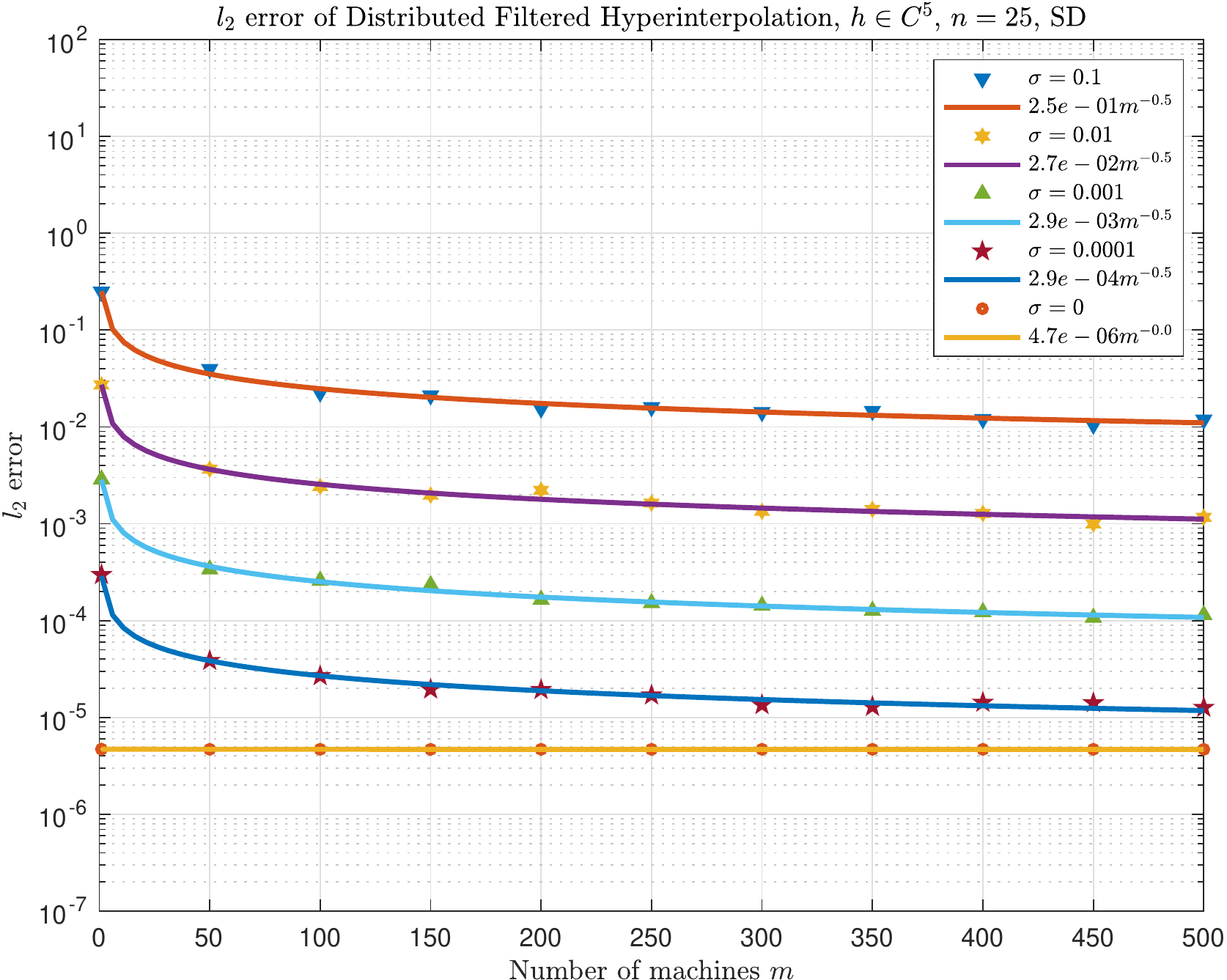}
    \end{minipage}
  \caption{Errors v.s. number of machines $m$ for distributed filtered hyperinterpolation, $n=25$, $m
  \leq 500, \sigma=0,0.0001,0.001,0.01,0.1$}
  \label{fig:errvsm}
\end{figure}
In this section, we test distributed filtered hyperinterpolation on noisy data on $\sph{2}$.
We use Womersley's \emph{symmetric spherical $t$-designs}\footnote{\url{https://web.maths.unsw.edu.au/\%7Ersw/Sphere/EffSphDes/}} \cite{Womersley2018,DeGoSe1977} as quadrature rule for distributed filtered hyperinterpolation. The symmetric spherical $t$-design is an equal-weighted quadrature with $\bigo{}{t^2}$ nodes satisfying \eqref{eq:quadrature} for degree $n\leq t$, where the order $\bigo{}{t^2}$ is optimal as a consequence of \cite{BoRaVi2013}. 

Let $(u)_{+}:=\max\{u,0\}$ for $u\in\Rone$. The normalised Wendland-Wu function is given by
\begin{equation*}
    \fnWend{}(u) := \fWend{}\left(\frac{8u}{15\sqrt{\pi}}\right),\quad u\in \Rone,
\end{equation*}
where $\fWend{}(u)$ is the original Wendland-Wu function
\begin{equation*}
  \fWend{}(u) := 
  (1-u)_{+}^{8}(32u^3 + 25u^2 + 8u + 1).
  \end{equation*}
See \cite{ChSlWo2014,Wendland1995,Wu1995}.
The $\phi(\bx\cdot\bz)\in H^{4.5}(\sph{2})$ is a radial basis function on the sphere $\sph{2}$ with centre at $\bz\in\sph{2}$ \cite{LeSlWe2010,NaWa2002}. We then define 
\begin{equation}\label{eq:rbf}
	f(\bx):=\sum_{i=1}^{6}\phi(\bx\cdot\bz_i)
\end{equation}
which is the linear combination of radial basis functions $\phi(\bx\cdot\bz_i)$ with centres at $\bz_1=(1,0,0)$, $\bz_2=(-1,0,0)$, $\bz_3=(0,1,0)$, $\bz_4=(0,-1,0)$, $\bz_5=(0,0,1)$, $\bz_6=(0,0,-1)$. The smoothness of $f$ is $4.5$, i.e. $f\in H^{4.5}(\sph{2})$. We use the function $f$ plus Gaussian white noise as the noisy data. That is, at each node $\bx_i\in\sph{2}$,
\begin{equation}\label{eq:noisydata}
	y_i = f(\bx_i) + \epsilon,\quad \epsilon\sim N(0,\sigma^2),
\end{equation}
where $\bx_i$ are the nodes of a symmetric spherical $t$-design and $\sigma\geq 0$.

Figures~\ref{fig:rbf}a and b show the pictures of $f$ and a realisation of $f+\epsilon$ with noise level $\sigma=0.1$. Figure~\ref{fig:rbf}c shows the distributed filtered hyperinterpolation approximation $\overline{f^\diamond_{D,n}}$ of degree $n=25$ for noisy data $y_i$. The distributed filtered hyperinterpolation uses $m=100$ machines and $C^5$-filter $\eta$ \cite{Wang2016,WaLeSlWo2017,WaSlWo2018} which satisfies the condition of smoothness of filter in Theorem~\ref{Theorem:distributed learning rate fixed design}. On the $j$-th machine, $j=1,\dots,100$, the filtered hyperinterpolation uses a $3n=75$-design with $2,852$ nodes, where the design is rotated from Womersley's symmetric spherical $75$-design \cite{Womersley2018} by the rotation matrix
\begin{equation*}
	\rho_j:=\begin{pmatrix}
	\cos(\theta_j) & -\sin(\theta_j) & 0\\
	\sin(\theta_j) & \cos(\theta_j) & 0\\
	0 & 0 & 1
\end{pmatrix}
\end{equation*}
with $\theta_j=j\pi/m$. The rotated spherical design satisfies the same polynomial exactness property for numerical integration as the unrotated symmetric spherical design since the rotation of a spherical design is still a spherical design with the same \emph{separation} and \emph{filling radii} \cite{Brauchart_etal2015,Brauchart_etal2018,GrSl2013}. The distributed filtered hyperinterpolation then satisfies the condition of Theorem~\ref{Theorem:distributed learning rate fixed design}, which uses $285,200$ points in total. The function values are evaluated at $10,000$ generalized spiral points on $\sph{2}$, which are equally-distributed points \cite{Bauer2000,RaSaZh1994}.
 Figure~\ref{fig:rbf}d shows the approximation error of $\overline{f^\diamond_{D,n}}$ to $f$, which illustrates that errors are small compared to the magnitude of the function $f$. 

Figure~\ref{fig:errvsL} shows the convergence rate (with respect to $n$) of the approximation error of the distributed filtered hyperinterpolation for $y_i$ with noise standard variance $\sigma=0,0.0001,0.001,0.01,0.1$. It illustrates that the approximation rate increases as the noise level becomes higher (i.e. when $\sigma$ is larger). When $\sigma=0$ and then the data is not contaminated by noise, the approximation error reaches the highest convergence rate at order $n^{-6.7}$. The convergence rate of the distributed filtered hyperinterpolation is higher than the upper bound of Theorem~\ref{Theorem:distributed learning rate fixed design}, since the function $f$ is sufficiently smooth.

Figure~\ref{fig:errvssigma} shows that the approximation error of the distributed filtered hyperinterpolation of degree $n=25$ with $m=100$ machines converges at the rate $0.2\sigma^{0.9}$ with decrease of noise level $\sigma$. The noise level of data impacts the approximation precision of distributed filtered hyperinterpolation.

Figure~\ref{fig:errvsm} illustrates the impact of the number $m$ of machines on approximation capability as we can see for a noise level $\sigma>0$, the approximation error converges at rate around $m^{-0.5}$. This means that the noise level changes the absolute magnitude of approximation error but has little impact on the trend of approximation rate with the increase of number of machines.

\newpage
\section{Proofs}\label{Sec.Proof}
\subsection{Proof for Theorem~\ref{Theorem:random cubature}}

{Theorem~\ref{Theorem:random cubature} uses the norming set method
\cite{MhNaWa2001} to derive the probabilistic quadrature rule.} To
prove Theorem~\ref{Theorem:random cubature},  {we need the following
four lemmas. The first one is the  Nikolski\^{\i} inequality on the
sphere, which was proved in \cite{MhNaWa1999}.

\begin{lemma}\label{Lemma:Nikolskii}
Let $1\leq p<q\leq\infty,$ $n\geq1$ be an integer, and
$P\in\Pi_n^{d}$. Then
$$
           \|P\|_{L^q(\mathbb
          S^{d})}\leq
          \tilde{C}_1n^{\frac{d}{p}-\frac{d}{q}}\|P\|_{L^p(\mathbb
          S^{d})},
$$
where the constant $\tilde{C}_1>0$ depends only on $d,p,q$.
\end{lemma}

The second one is the following concentration inequality, which was
established in \cite{WuZh2005}.
\begin{lemma}\label{Lemma:Concentration inequality 1.1}
Let $\mathcal G$ be a set of functions on compact metric space $Z$. For every $g\in\mathcal G$, if $|g-\mathbf Eg|\leq B$ almost everywhere and $\mathbf E(g^2)\leq \tilde{c}(\mathbf E g )^\alpha$ for some $B\geq0$, $0\leq\alpha\leq 1$ and $\tilde{c}\geq0$. Then, for any $\varepsilon>0$,
\begin{equation*}
	\mathbf P\left\{\sup_{g\in\mathcal G}\frac{\left|\mathbf
     Eg-\frac1m\sum_{i=1}^mg(z_i)\right|}{\sqrt{(\mathbf
     Eg)^\alpha+\varepsilon^\alpha }}>\varepsilon^{1-\frac{\alpha}2}\right\}\leq2\mathcal N(\mathcal
     G,\varepsilon)\exp\left\{-\frac{m\varepsilon^{2-\alpha}}{2(\tilde{c}+\frac13B\varepsilon^{1-\alpha})}\right\},
\end{equation*}
where $\mathcal N(\mathcal G,\varepsilon)$ denotes the covering
number \cite{WuZh2005} of $\mathcal G$ with radius $\varepsilon$.
\end{lemma}

The third one is a covering number estimate for Banach spaces, as given in} \cite{ZhJe2006}.
\begin{lemma}\label{Lemma:Covering number}
Let $\mathbb B$ be a finite-dimensional Banach space.
 Let $B_R$ be the closed ball of radius $R$
centered at the origin given by
 $B_R:=\{f\in\mathbb B:\|f\|_{\mathbb B}\leq R\}$. Then,
\begin{equation*}
	\log\mathcal N(B_R,\varepsilon)\leq
            \dim(\mathbb B) \log\left(\frac{4R}\varepsilon\right).
\end{equation*}
\end{lemma}

To state the last lemma, we need  following definitions. Let
$\mathcal X$ be a finite dimensional vector space with norm
$\|\cdot\|_{\mathcal X}$, and $\mathcal Z\subset \mathcal X^*$ be a
finite set. We say that $\mathcal Z$ is a \emph{norm generating set} for
$\mathcal X$ if the mapping $T_{\mathcal Z}: \mathcal
X\rightarrow\mathbb R^{|\mathcal Z|}$ defined by $T_{\mathcal
Z}(x)=(z(x))_{z\in \mathcal Z}$ is injective, and $T_{\mathcal Z}$
is called the \emph{sampling operator}. Let $W:=T_{\mathcal Z}(\mathcal X)$ be
the range of $T_{\mathcal Z}$, then the
 injectivity of $T_{\mathcal Z}$ implies that $T_{\mathcal Z}^{-1}:W\rightarrow \mathcal X$ exists.
 Let $\|\cdot\|_{\mathbb R^{|\mathcal Z|}}$ be the norm of $\mathbb R^{|\mathcal Z|}$ norm, and
 $\|\cdot\|_{\mathbb R^{|\mathcal Z|^*}}$ the dual norm on $\mathbf
 R^{|\mathcal Z|^*}$ for $\|\cdot\|_{\mathbb R^{|\mathcal Z|}}$. Equipping $W$ with the induced norm, and let
 $\|T_{\mathcal Z}^{-1}\|:=\|T_{\mathcal Z}^{-1}\|_{W\rightarrow \mathcal
 X}.$  In addition, let
 $\mathcal K_+$ be the positive cone of $\mathbb R^{|\mathcal Z|}$  {which is the set of all}
 $(r_z)_{z\in \mathcal{Z}}\in\mathbb R^{|\mathcal Z|}$  {such that} $r_z\geq0$. Then
 the following lemma \cite{MhNaWa2001} holds.

\begin{lemma}\label{Lemma:norming set}
Let $\mathcal Z$ be a norm generating set for $\mathcal X$, with
$T_{\mathcal Z}$ the corresponding sampling operator. If $g\in
\mathcal X^*$ with $\|g\|_{\mathcal X^*}\leq \mathcal A$, then there
exist positive numbers $\{a_z\}_{z\in \mathcal Z}$, depending only
on $g$ such that for every $x\in\mathcal X,$
\begin{equation*}
	g(x)=\sum_{z\in \mathcal Z}a_zz(x),
\end{equation*}
and
\begin{equation*}
	\|(a_z)\|_{\mathbb R^{|\mathcal Z|^*}}\leq \mathcal A\|T_{\mathcal Z}^{-1}\|.
\end{equation*}
Also, if $W$ contains an interior point $v_0\in \mathcal K_+$ and if
$g(T_{\mathcal Z}^{-1}v)\geq0$ when $v\in W\cap \mathcal K_+$, then
we may choose $a_z\geq 0.$
\end{lemma}

We are ready to prove Theorem~\ref{Theorem:random cubature}.
\begin{proof}[Proof of Theorem~\ref{Theorem:random cubature}]
For $p=1,2$, without loss of generality, we prove Theorem~\ref{Theorem:random cubature} for $P_n\in\Pi_n^d$ satisfying $\|P_n\|_{p,\mu}= A$ for some constant $A>0$. For an arbitrary $P_n\in\Pi_{n}^{d}$ with $\|P_n\|_{p,\mu}= A$, it follows from \eqref{condition on distribution} and Lemma~\ref{Lemma:Nikolskii} that
\begin{equation*}
	\|P_n\|_{L_\infty(\mathbb S^d)}\leq
               \tilde{C}_1n^{\frac{d}{p}}\|P_n\|_{L_p(\mathbb S^d)}
               \leq c^{1/p}_4\tilde{C}_1n^{\frac{d}{p}}\|P_n\|_{p,\mu},
\end{equation*}
 and
\begin{eqnarray*}
  \mathbf E\left\{|P_n|^{2p}\right\}=\int_{\mathbb S^d}|P_n(\mathbf x)|^{2p}{\rm d}\mu(\mathbf x)\leq \|P_n\|^p_{L_\infty(\mathbb S^d)}
              \int_{\mathbb S^d}|P_n(\mathbf x)|^p{\rm d}\mu(\mathbf x)
              \leq
              c_4(\tilde{C}_1)^pn^d\|P_n\|^p_{p,\mu} \mathbf E\left[|P_n|^p\right].
\end{eqnarray*}
Let $Z=\sph{d}$, $g(z_i)=|P_n(\mathbf x_i)|^p$, $B=2c_4(\tilde{C}_1)^pn^d A^p$, $\tilde{c}=c_4(\tilde{C}_1)^pn^d A^p$, $m=N$, $\alpha=1$ and $\mathcal G_p=\{|P_n|^p:P_n\in\Pi_n^d,\|P_n\|_{p,\mu}=A\}$ in Lemma~\ref{Lemma:Concentration inequality 1.1}. Then, for any $\varepsilon>0$,
\begin{align*}
     &\mathbf P\left\{\sup_{P_n\in\Pi_n^{d},\|P_n\|_{p,\mu}=A}
     \frac{\left|\|P_n\|_{p,\mu}^p-\frac1{N}\sum_{i=1}^{N}
     |P_n(\mathbf x_i)|^p\right|}{
     \sqrt{\|P_n\|_{p,\mu}^p+\varepsilon}}>\sqrt{\varepsilon}\right\}\\
     &\quad \leq 2\mathcal N(\mathcal G_p,\varepsilon)
     \exp\left\{
     -\frac{{N}\varepsilon}{\tilde{C}_2n^dA^p}\right\},
\end{align*}
where $\tilde{C}_2=10c_4(\tilde{C}_1)^p/3$. For $p=1$, we have
$|P_n|-|P_n^*|\leq|P_n-P_n^*|$ for any $P_n,P_n^*\in\Pi_n^d$. Then,
it follows from the definition of the covering number that
$\mathcal N(\mathcal G_1,\varepsilon)\leq \mathcal N(\mathcal
      G_1',\varepsilon)$,
where $\mathcal G_1':=\{P_n\in\Pi_n^d:\|P_n\|_{p,\mu}=A\}$. For
$p=2$, $|P_n|^2\in\Pi_{2n}^d$. Let $\mathcal G_2':=\{P_n\in\Pi_{2n}^d:\|P_n\|_{p,\mu}=A\}$. Then, $\mathcal N(\mathcal G_2,\varepsilon)=\mathcal N(\mathcal
      G_2',\varepsilon)$.
It then follows from Lemma~\ref{Lemma:Covering
number} that for $p=1,2$,
\begin{align*}
     &\mathbf P\left\{\sup_{P_n\in\Pi_n^{d},\|P_n\|_{p,\mu}=A}
     \frac{\left|\|P_n\|_{p,\mu}^p-\frac1{N}\sum_{i=1}^{N}
     |P_n(\mathbf x_i)|^p\right|}{
     \sqrt{\|P_n\|_{p,\mu}^p+\varepsilon}}>\sqrt{\varepsilon}\right\}\\
     &\quad\leq 2\exp\left\{(2n)^d\log\frac{4A^p}{\varepsilon}
     -\frac{{N}\varepsilon}{\tilde{C}_2n^dA^p}\right\},
\end{align*}
where we use $\dim \mathcal{G}_p\leq (pn)^d$ for $p=1,2$.
Let
$\varepsilon=A^p/4$. As $N/n^d>c$ for a sufficiently large $c>0$, 
\begin{equation*}
	(2n^d)\log \frac{4A^p}{\varepsilon}<\frac{N\varepsilon}{\tilde{C}_2 n^d A^p}.
\end{equation*}
Then, with confidence
\begin{equation}\label{eq:confidence_a}
	1-2\exp\left\{-\tilde{C}_3 N/n^d\right\},
\end{equation}
there holds
\begin{equation*}
	\left|\|P_n\|_{p,\mu}^p-\frac1{N}\sum_{i=1}^{N}|P_n(\mathbf x_i)|^p\right|
     \leq \sqrt{\varepsilon(\|P_n\|^p_{p,\mu}+\varepsilon)}=\frac{\sqrt{5}}{4}\|P_n\|^p_{p,\mu}.
\end{equation*}
Then, with the same confidence as \eqref{eq:confidence_a},
\begin{equation}\label{MZ}
     \frac{1}3\|P_n\|_{p,\mu}^p\leq\frac1{N}\sum_{i=1}^{N}
           |P_n(\mathbf x_i)|^p\leq\frac{5}3\|P\|_{p,\mu}^p\quad
           \forall P_n \in\Pi_n^{d},\ p=1,2.
\end{equation}

Now, we use \eqref{MZ} with $p=2$ and Lemma~\ref{Lemma:norming set}
to prove Theorem~\ref{Theorem:random cubature}. In Lemma
\ref{Lemma:norming set}, we take $\mathcal X=\Pi_n^{d}$,
$\|P_n\|_{\mathcal X}=\|P_n\|_{2,\mu}$, and $\mathcal Z$ to be the
set of point evaluation functionals $\{\delta_{\mathbf
x_i}\}_{i=1}^{N}$. The operator $T_{\mathcal Z}$ is then the
restriction map $P_n\mapsto P_n|_{X_N} $ and
\begin{equation*}
	\|f\|_{X_N,2}:= \left(\frac1N\sum_{i=1}^N|f(\mathbf x_i)|^2\right)^\frac12.
\end{equation*}
It follows from \eqref{MZ} that with confidence at least
\begin{equation*}
	1-2\exp\left\{-\tilde{C}_3 N/n^d \right\},
\end{equation*}
there holds $\|T_{\mathcal Z}^{-1}\|\leq \sqrt{\frac{5}{3}}$. We take $g$ to be the functional
\begin{equation*}
	g: P_n\mapsto \int_{\mathbb S^d}P_n(x)\mathrm{d}\mu(x).
\end{equation*}
By H\"{o}lder inequality, $\|g\|_{\mathcal X^*}\leq 1$.
Lemma~\ref{Lemma:norming set} then shows that there exists a set of real numbers $\{w_{i,n}\}_{i=1}^N$ such that
\begin{equation*}
	\int_{\mathbb S^d}P_n(x)\mathrm{d}\mu(x)=\sum_{i=1}^{N}w_{i,n}P_n(\mathbf x_i)\quad \forall P_n\in \Pi_n^d.
\end{equation*}
holds, and $\frac1N\sum_{i=1}^{N}\left(\frac{ w_{i,n} }{1/{N}}\right)^2\leq 2$, with confidence at least $1-2\exp\left\{-\tilde{C}_3 N/n^d \right\}$.

Finally, we use the second assertion of Lemma~\ref{Lemma:norming
set} and \eqref{MZ} with $p=1$ to prove  the positivity of
$w_{i,n}$. Since $1\in \Pi_n^d$, we have
$v_0:=1|_{X_N}=(1,1,\dots,1)$ is an interior point of $\mathcal
K_+$. For $P_n\in\Pi_n^d$, $T_{\mathcal
Z}P_n=P_n|_{X_N}$ is in $W\cap\mathcal K_+$ if and only if
$P_n(\mathbf x_i)\geq 0$ for all $\mathbf x_i\in X_N$. For an arbitrary $P_n$ satisfying $P_n(\mathbf x_i)\geq0$ with $\mathbf x_i\in X_N$, define $\xi_i(P_n)=P_n(\mathbf x_i)$. From Lemma~\ref{Lemma:Nikolskii} and \eqref{condition on distribution}, we obtain the following estimates: for $i=1,\dots,N$,
\begin{align*}
	|\xi_i| &\leq \|P_n\|_{L_\infty(\mathbb S^d)}\leq
         \tilde{C}_1n^d\|P_n\|_{L_1(\mathbb S^d)}\leq
         \tilde{C}_1c_4n^d\|P_n\|_{1,\mu},\\
    |\xi_i-\mathbf E\xi_i| &\leq
       2\|P_n\|_{L_\infty(\mathbb S^d)}\leq 2\tilde{C}_1c_4n^d\|P_n\|_{1,\mu},\\
	\mathbf  E\xi_i^2 &\leq \|P_n\|_{L_\infty(\mathbb S^d)}
      \|P_n\|_{1,\mu}\leq \tilde{C}_1 c_4 n^d \|P_n\|^2_{1,\mu}.
\end{align*}
Applying Lemma~\ref{Lemma:Concentration inequality 1.1} with
$B=2 \tilde{C}_1 c_4 n^d A$, $\tilde{c}=\tilde{C}_1 c_4 n^d A^2$ and $\alpha=0$ to the set $\{P_n:P_n\in\Pi_n^d, \|P_n\|_{1,\mu}=A\}$, by Lemma~\ref{Lemma:Covering number}, we obtain for any $\varepsilon>0$,
\begin{align*}
     &\mathbf P\left\{\sup_{P_n\in\Pi_n^{d}, 
     P_n|_{X_N}\geq0, \|P_n\|_{1,\mu}=A}
     \left| g(P_n)-\frac1{N}
     \sum_{i=1}^{N} P_n(\mathbf x_i) \right|>\varepsilon\right\}\\
     &\quad\leq
     2\exp\left\{n^d\log\frac{4A }{\varepsilon}
     -\frac{{N}\varepsilon^2}{2\tilde{C}_1c_4n^dA (A +2\varepsilon/3)}\right\}.
\end{align*}
Let $\varepsilon=A/4$. We then obtain that with
confidence
\begin{equation*}
	1-2\exp\left\{-\tilde{C}_4 N/n^d\right\},
\end{equation*}
there holds
\begin{equation*}
	 \left|g(P_n)-\frac1{N}
     \sum_{i=1}^{N} P_n(\mathbf x_i)\right|\leq
     \frac14\|P_n\|_{1,\mu} \quad\forall P_n\in \Pi_n^d.
\end{equation*}
This and \eqref{MZ} imply that, with confidence $1-4\exp\bigl\{-C N/n^d\bigr\}$ (where $C$ depends only on $\tilde{C}_3$ and $\tilde{C}_4$), for any $P_n$ satisfying $P_n(\mathbf x_i)\geq0$ $\forall\mathbf x_i\in X_N$, the inequality
\begin{equation*}
	\left|g(P_n)-\frac1{N}\sum_{i=1}^{N} P_n(\mathbf x_i)\right|\leq 
     \frac{3}4\frac1{N}\sum_{i=1}^{N} P_n(\mathbf x_i)
\end{equation*}
holds, then,
\begin{equation*}
	g(P_n)\geq \frac{1}4 \frac1{N}\sum_{i=1}^{N}
             P_n(\mathbf x_i) \geq0.
\end{equation*}
It hence follows from Lemma~\ref{Lemma:norming
set} that $w_{i,n}\geq 0$, thus completing the proof of Theorem~\ref{Theorem:random cubature}.
\end{proof}

\subsection{Proofs for the theorems in Section~\ref{Sec.Fixed design}}
The following lemma shows that the filtered kernel has the following
localization property, as proved by \cite{NaPeWa2006} and
also \cite{Wang2016,WaLeSlWo2017,WaSlWo2018}.
\begin{lemma}[\cite{NaPeWa2006}]\label{Lemma:Localization}
Let $d\geq 2$ and $\eta$ be a filter in $C^\kappa(\mathbb
R_+)$ with $1\leq \kappa<\infty$ such that $\eta$ is constant
on $[0,a]$ for some $0<a<2$. Then,
\begin{equation*}
    \left|K_n(\cos\theta)\right|
              \leq \frac{c n^{d}}{(1+n\theta)^\kappa},\quad n\geq1,
\end{equation*}
where $c$ is a constant depending only on $d,\eta$ and $\kappa$.
\end{lemma}

Lemma~\ref{Lemma:Localization} gives the following upper bound of
the $L_p$ norm of the filtered kernel.

\begin{lemma}\label{Lemma:bound for filtered norm}
Let $d\geq 2$, $1\leq p\leq \infty$  and $\eta$ be a filter
in $C^\kappa(\mathbb R_+)$ with
$\kappa\geq\left\lfloor\frac{d+3}2\right\rfloor$ such that
$\eta$ is constant on $[0,a]$ for some $0<a<2$. Then
\begin{equation*}
    \|K_n(\bx\cdot \cdot)\|_{L_p(\mathbb S^d)}\leq
        c_1 n^{d(1-1/p)}\quad \forall \bx\in\mathbb S^d,\; n\geq1,
\end{equation*}
where $c_1$ is a constant depending only on $d,\eta,\kappa$ and $p$.
\end{lemma}

The above lemma for $p=1$ was proved by \cite{WaLeSlWo2017} (see
also \cite{NaPeWa2006} for $\kappa\geq d+1$).  {The case $p> 1$} can
be  {obtained from} the case $p=1$  {with the fact that}
$K_n\in\Pi_{2n}^d$ and the Nikolski\^{\i} inequality for spherical
polynomials \cite{MhNaWa1999}.

\begin{proof}[Proof of Theorem~\ref{Theorem:learning rate fixed design}]
Define
\begin{equation}\label{def.semi 1}
            f^{\diamond,*}_{D,n}(x):=\sum_{i=1}^{|D|}w_{i,s,D}f^*(\mathbf x_i)
            K_n(\mathbf x_i\cdot \mathbf x).
\end{equation}
As $\mathbf E\{\epsilon_i\}=0$ for any $i=1,\dots,|D|$,
\begin{align*}
           \mathbf E\left\{f^\diamond_{D,n}(x)\right\}
           &=
           \mathbf E\left\{\sum_{i=1}^mw_{i,s,D}y_iK_n(\mathbf x_i\cdot \mathbf x)\right\}
            =
           \mathbf E\left\{\sum_{i=1}^mw_{i,s,D}(f^*(\mathbf x_i)+\epsilon_i)K_n(\mathbf x_i\cdot \mathbf x)\right\}\\
           &=
           \sum_{i=1}^mw_{i,s,D}f^*(\mathbf x_i)K_n(\mathbf x_i\cdot \mathbf x)
           +\sum_{i=1}^mw_{i,s,D}\mathbf E\{\epsilon_i\}K_n(\mathbf x_i\cdot \mathbf x)
           =
           f^{\diamond,*}_{D,n}(x),
\end{align*}
then,
\begin{equation}\label{semipopulation 1}
           \mathbf E\left\{f^{\diamond,*}_{D,n}(x)-f^\diamond_{D,n}(x) \right\}=0.
\end{equation}
This implies
\begin{align}\label{decomposition 1.1}
   &\mathbf E\left\{\|f^\diamond_{D,n}-f^*\|^2_{L_2(\mathbb S^d)} \right\}\\
     &\quad=
     \int_{\mathbb S^d} \mathbf E\{(f^*(x)-f^\diamond_{D,n}(x))^2
     \}\mathrm{d}\omega(\mathbf x)   \nonumber\\
    &\quad =
     \int_{\mathbb S^d} \mathbf E\{(f^*(x)-f^{\diamond,*}_{D,n}(x)
     +f^{\diamond,*}_{D,n}(x)-f^\diamond_{D,n}(x))^2
     \}\mathrm{d}\omega(\mathbf x)  \nonumber\\
    &\quad= \int_{\mathbb S^d}  (f^{\diamond,*}_{D,n}(x)-f^*(x))^2 \mathrm{d}\omega(\mathbf x)
    + \int_{\mathbb S^d} \mathbf
    E\{(f^{\diamond,*}_{D,n}(x)-f^\diamond_{D,n}(x))^2 \}\mathrm{d}\omega(\mathbf x)
         \nonumber\\
         & \quad:=
         \mathcal{A}^\diamond_{D,n} + \mathcal{S}^\diamond_{D,n}.\nonumber
\end{align}
Lemma~\ref{Lemma:Filtered h app} gives
\begin{equation}\label{bound A.1}
    \mathcal A^\diamond_{D,n}\leq
    c_5^2 \:n^{-2r}\|f^*\|^2_{\mathbb W_2^r(\mathbb
     S^d)}.
\end{equation}
To bound $\mathcal S^\diamond_{D,n}$, we observe from \eqref{eq:noisydats} that
\begin{align*}
         \mathbf
         E\left\{(f^{\diamond,*}_{D,n}(\mathbf x)-f^\diamond_{D,n}(\mathbf x))^2 \right\}
           &=
          \mathbf E\left\{\left(\sum_{i=1}^{|D|}(y_i-f^*(\mathbf x_i))w_{i,s,D}
          K_n(\mathbf x_i\cdot \mathbf x)\right)^2 \right\}\\
          &=
          \mathbf E\left\{\left(\sum_{i=1}^{|D|}\epsilon_iw_{i,s,D}
          K_n(\mathbf x_i\cdot \mathbf x)\right)^2 \right\}\\
          &\leq M^2 \sum_{i=1}^{|D|}w_{i,s,D}^2|K_n(\mathbf x_i\cdot \mathbf x)|^2,
\end{align*}
where the last inequality uses the independence of $\epsilon_1,\dots,\epsilon_{|D|}$.
This together with Lemmas \ref{Lemma:bound for filtered norm} and \ref{Lemma:fixed cubature} shows
\begin{align}\label{bound S.1}
        \mathcal S^\diamond_{D,n}
         &\leq
         M^2 \int_{\mathbb S^d} \sum_{i=1}^{|D|} w_{i,s,D}^2|
         K_n(\mathbf x_i\cdot \mathbf x)|^2 \mathrm{d}\omega(\mathbf x) \nonumber\\
          &=
          M^2 \sum_{i=1}^{|D|} w_{i,s,D}^2\int_{\mathbb S^d}|K_n(\mathbf x_i\cdot \mathbf x)|^2 \mathrm{d}\omega(\mathbf x)
           \leq
          c_1M^2n^d \sum_{i=1}^{|D|} w_{i,s,D}^2 \leq
          \frac{c_1c_2^2M^2n^d}{|D|}.
\end{align}
Putting \eqref{bound S.1} and \eqref{bound A.1} to
\eqref{decomposition 1.1}, we obtain
\begin{equation}\label{IMportant.1}
     \mathbf E\left\{\|f^\diamond_{D,n}-f^*\|^2_{L_2(\mathbb S^d)}\right\}
     \leq
     c_5^2n^{-2r}\|f^*\|^2_{\mathbb W_2^r(\mathbb
     S^d)}+ \frac{c_1c_2^2M^2n^d}{|D|},
\end{equation}
with $\frac{c_3}6 |D|^{\frac1{2r+d}}\leq n\leq \frac{c_3}3 |D|^{\frac1{2r+d}}$, then,
\begin{equation*}
	\mathbf E\left\{\|f^\diamond_{D,n}-f^*\|^2_{L_2(\mathbb S^d)}\right\}\leq C_1|D|^{-\frac{2r}{2r+d}}
\end{equation*}
with $C_1:= 36^r c_5^2c_3^{-2r} \|f^*\|^2_{\mathbb W_2^r(\mathbb S^d)} + 3^{-d}c_1 c_2^2 c_3^d M^2$, thus completing the proof.
\end{proof}

To prove Theorem~\ref{Theorem:distributed learning rate fixed
design}, we need the following lemma, which is a modified version of
\cite[Proposition 4]{GuLiZh2017}.
\begin{lemma}\label{Lemma:distributed.1}
 For $\overline{f^\diamond_{D,n}}$ in Definition~\ref{defn:distrfih}, there holds
\begin{align}
        &\mathbf E\left\{\bigl\|\overline{f^\diamond_{D,n}}-f^*\bigr\|_{L_2(\mathbb S^d)}^2\right\}\notag\\
        &\quad\leq
         \sum_{j=1}^m\frac{|D_j|^2}{|D|^2}\mathbf E\left\{\|f^\diamond_{D_j,n}
         -f^*\|_{L_2(\mathbb
         S^d)}^2\right\}
        +\sum_{j=1}^m\frac{|D_j|}{|D|}
        \left\| f^{\diamond,*}_{D_j,n}-f^*\right\|_{L_2(\mathbb
        S^d)}^2,\label{eq:distributed.bound}
\end{align}
where $f^{\diamond,*}_{D_j,n}$ is given by \eqref{def.semi 1}.
\end{lemma}

\begin{proof} Due to \eqref{distrilearn 1} and
$\sum_{j=1}^m\frac{|D_j|}{|D|}=1$, we have
\begin{align*}
       &\bigl\|\overline{f^\diamond_{D,n}}-f^*\bigr\|_{L_2(\mathbb S^d)}^2
        =
       \left\|\sum_{j=1}^m\frac{|D_j|}{|D|}(f^\diamond_{D_j,n}-f^*)\right\|_{L_2(\mathbb S^d)}^2\\
       &\hspace{-2mm}=
       \sum_{j=1}^m\frac{|D_j|^2}{|D|^2}\|f^\diamond_{D_j,n}-f^*\|_{L_2(\mathbb S^d)}^2
       +
       \sum_{j=1}^m\frac{|D_j|}{|D|}\left\langle
       f^\diamond_{D_j,n}-f^*,\sum_{k\neq
       j}\frac{|D_k|}{|D|}(f^\diamond_{D_k,n}-f^*)\right\rangle_{L_2(\mathbb S^d)}.
\end{align*}
Taking expectations gives
\begin{align*}
        &\mathbf E\left\{\bigl\|\overline{f^\diamond_{D,n}}-f^*\bigr\|_{L_2(\mathbb
        S^d)}^2\right\}
        =
        \sum_{j=1}^m\frac{|D_j|^2}{|D|^2}\mathbf E\left\{\|f^\diamond_{D_j,n}
        -f^*\|_{L_2(\mathbb S^d)}^2\right\}\\
        &\quad+
        \sum_{j=1}^m\frac{|D_j|}{|D|}\left\langle
       \mathbf E_{D_j}\bigl\{f^\diamond_{D_j,n}\bigr\}-f^*,\mathbf E\left\{\overline{f^\diamond_{D,n}}\right\}-f^*- \frac{|D_j|}{|D|}
       \left(\mathbf E_{D_j}\{f^\diamond_{D_j,n}\}-f^*\right)\right\rangle_{L_2(\mathbb S^d)},
\end{align*}
where
\begin{align*}
       &\sum_{j=1}^m\frac{|D_j|}{|D|}\left\langle
       \mathbf E_{D_j}\{f^\diamond_{D_j,n}\}-f^*,\mathbf E\left\{\overline{f^\diamond_{D,n}}\right\}-f^*\right\rangle_{L_2(\mathbb S^d)}\\
      	&\qquad=
       \mathbf E\left\{\left\langle
        \overline{f^\diamond_{D,n}}-f^*,\mathbf E\left\{\overline{f^\diamond_{D,n}}\right\}-f^*\right\rangle_{L_2(\mathbb S^d)}\right\}
        =
        \left\|\mathbf E\left\{\overline{f^\diamond_{D,n}}\right\}-f^*\right\|_{L_2(\mathbb S^d)}^2.
\end{align*}
Then,
\begin{align*}
        &\mathbf
        E\left\{\bigl\|\overline{f^\diamond_{D,n}}-f^*\bigr\|_{L_2(\mathbb
        S^d)}^2\right\}
          =
         \sum_{j=1}^m\frac{|D_j|^2}{|D|^2}\mathbf E\left\{\|f^\diamond_{D_j,n}-f^*\|_{L_2(\mathbb
         S^d)}^2\right\}\\
         &\qquad-\sum_{j=1}^m\frac{|D_j|^2}{|D|^2}\left\|\mathbf E\{f^\diamond_{D_j,n}\}-f^*\right\|_{L_2(\mathbb S^d)}^2 +
         \left\|\mathbf E\left\{\overline{f^\diamond_{D,n}}\right\}-f^*\right\|_{L_2(\mathbb S^d)}^2.
\end{align*}
By \eqref{semipopulation 1},
$$
      \mathbf
      E\left\{\overline{f^\diamond_{D,n}}\right\}=\sum_{j=1}^m\frac{|D_j|}{|D|}f^{\diamond,*}_{D_j,n}.
$$
This plus
$\sum_{j=1}^m\frac{|D_j|}{|D|}=1$ gives
\begin{align*}
    \left\|\mathbf E\left\{\overline{f^\diamond_{D,n}}\right\}-f^*\right\|^2_{L_2(\mathbb S^d)}
    &=
       \left\|\sum_{j=1}^m\frac{|D_j|}{|D|}\left(f^{\diamond,*}_{D_j,n}-f^*\right)\right\|_{L_2(\mathbb S^d)}^2\\[1mm]
    &\leq
       \sum_{j=1}^m\frac{|D_j|}{|D|}\bigl\|f^{\diamond,*}_{D_j,n}-f^*\bigr\|_{L_2(\mathbb S^d)}^2,
\end{align*}
thus proving the bound in \eqref{eq:distributed.bound}.
\end{proof}

\begin{proof}[Proof of Theorem~\ref{Theorem:distributed learning rate fixed
design}]
By Lemma~\ref{Lemma:distributed.1}, we
only need to estimate the bounds of $\mathbf
E\left\{\|f^\diamond_{D_j,n}-f^*\|_{L_2(\mathbb
         S^d)}^2\right\}$ and $\bigl\|f^{\diamond,*}_{D_j,n}-f^*\bigr\|_{L_2(\mathbb
         S^d)}^2$.
Since $\min_{j=1,\dots,m}|D_j|\geq |D|^{d/(2r+d)}$ and $D_j$ is
$\tau$-quasi uniform, it follows from Lemma~\ref{Lemma:fixed
cubature} that there exists a quadrature rule for each local
machine which is exact for polynomials of degree $3n-1$ for $n\leq \frac{c_3}3|D|^{1/(2r+d)}$.
From \eqref{IMportant.1} with $D=D_j$, for $j=1,\dots,m$,
$$
     \mathbf E\left\{\|f^\diamond_{D_j,n}-f^*\|^2_{L_2(\mathbb S^d)}\right\}
     \leq
     c_5^2n^{-2r}\|f^*\|^2_{\mathbb W_2^r(\mathbb
     S^d)}+ \frac{c_1c_2^2M^2n^d}{|D_j|}.
$$
This together with $\sum_{i=1}^m\frac{|D_j|}{|D|}=1$ gives
\begin{align}\label{bound first term.1}
     &\sum_{j=1}^m\frac{|D_j|^2}{|D|^2}\mathbf E\left\{\|f^\diamond_{D_j,n}-f^*\|_{L_2(\mathbb
         S^d)}^2\right\}  \nonumber\\
         &\quad\leq 36^r c_5^2c_3^{-2r}\|f^*\|^2_{\mathbb W_2^r(\mathbb S^d)}|D|^{-\frac{2r}{2r+d}}+ 3^{-d}c_1 c_2^2 c_3^d M^2
        \sum_{j=1}^m\frac{|D_j|^2}{|D|^2}\frac{|D|^{\frac{d}{2r+d}}}{|D_j|}
        = C_1|D|^{-\frac{2r}{2r+d}},
\end{align}
where $C_1:= 36^rc_5^2c_3^{-2r}\|f^*\|^2_{\mathbb W_2^r(\mathbb
     S^d)} + 3^{-d}c_1c_2^2c_3^d M^2$. 
 
 For each $j=1,\dots,m$, $D_j$ is $\tau$-quasi uniform. Lemma~\ref{Lemma:fixed cubature} implies that there exists a quadrature rule with nodes of $D_j$ and $|D_j|$ positive weights such that $f^*_{D_j,n}$ is a filtered hyperinterpolation for the noise-free data set $\{\mathbf x_i,f^*(\mathbf x_i)\}_{\mathbf x_i\in D_j}$. Lemma~\ref{Lemma:Filtered h app} then gives
\begin{equation*}
	\left\| f^{\diamond,*}_{D_j,n} -f^*\right\|_{L_2(\mathbb
         S^d)}^2
     \leq c_5^2n^{-2r}\|f^*\|^2_{\mathbb W_2^r(\mathbb
     S^d)} \quad \forall j=1,2,\dots,m.
\end{equation*}
This together with
$\sum_{j=1}^m\frac{|D_j|}{|D|}=1$ and $\frac{c_3}6 |D|^{\frac{1}{2r+d}}\leq n\leq \frac{c_3}3|D|^{\frac{1}{2r+d}}$ gives
\begin{equation}\label{bound second term.1}
     \sum_{j=1}^m\frac{|D_j|}{|D|}\left\| f^{\diamond,*}_{D_j,n}
     -f^*\right\|_{L_2(\mathbb S^d)}^2
     \leq
     36^r c_5^2c_3^{-2r}\|f^*\|^2_{\mathbb W_2^r(\mathbb
     S^d)}|D|^{-\frac{2r}{2r+d}}.
\end{equation}
Using \eqref{bound first term.1} and \eqref{bound second term.1}
in Lemma~\ref{Lemma:distributed.1},
\begin{equation*}
	\mathbf E\left\{\|\overline{f^\diamond_{D,n}}-f^*\|_{L_2(\mathbb S^d)}^2\right\}
     \leq C_2|D|^{-\frac{2r}{2r+d}}.
\end{equation*}
This then proves \eqref{error 1.1} with
$C_2 = 2^{2r+1}\cdot 9^{r}c_5^2c_3^{-2r}\|f^*\|^2_{\mathbb W_2^r(\mathbb S^d)} + 3^{-d}c_1 c_2^2 c_3^d M^2$.
\end{proof}

\subsection{Proofs for the theorems in Section~\ref{Sec.Random}}
\begin{proof}[Proof of Theorem~\ref{Theorem:learning rate fih random samp}]
Let $\{a_{i,n,D}\}_{i=1}^{|D|}$ be the real numbers computed in \eqref{eq:fihransamp}. Since $\{\mathbf x_i\}_{i=1}^{|D|}$ is a set of random points on $\sph{d}$, we define four events, as follows. Let $\Omega_{D}$ be the event such that
     $\sum_{i=1}^{|D|}|a_{i,n,D}|^2\leq\frac{2}{|D|}$ and $\Omega_{D}^c$ be
the complement of $\Omega_D$, i.e. $\Omega_{D}^c$ be the event
$\sum_{i=1}^{|D|}|a_{i,n,D}|^2>\frac{2}{|D|}$. Let $\Xi_D$
the event that $\{(a_{i,n,D},\bx_i)\}_{i=1}^{|D|}$ is a
quadrature rule exact for polynomials in $\Pi_n^d$ and $\Xi_D^c$   the complement event of $\Xi_D$. Then, by Theorem~\ref{Theorem:random cubature},
\begin{equation}\label{Probability for events}
     \mathbf P\{\Omega_{D}^c\}\leq\mathbf P\{\Xi_D^c\}\leq
     4\exp\left\{-C|D|/n^d\right\}.
\end{equation}
We write
\begin{align}
      &\mathbf E\left\{\|f_{D,n}-f^*\|^2_{L_2(\mathbb S^d)}\right\}\label{expectation decomposition}\\
      &\quad=\mathbf E\left\{\|f_{D,n}-f^*\|^2_{L_2(\mathbb S^{d})}|\Omega_{D}\right\}\mathbf P\{\Omega_{D}\}
      +\mathbf E\left\{\|f_{D,n}-f^*\|^2_{L_2(\mathbb S^{d})}|\Omega_{D}^c\right\}\mathbf P\{\Omega_{D}^c\}.\notag 
\end{align}
Under the event $\Omega_{D}^c$, we have from \eqref{eq:fihransamp} and \eqref{SFH for noisy data} that
$f_{D,n}(x)=0$. Then, by \eqref{Probability for events},
\begin{equation}\label{bound without event}
     \mathbf E\left\{\|f_{D,n}-f^*\|^2_{L_2(\mathbb S^d)}|\Omega_D^c\right\}\mathbf P\{\Omega_{D}^c\}
      \leq 4\|f^*\|^2_{L_\infty(\mathbb S^d)}\exp\bigl\{-C|D|/n^d\bigr\}.
\end{equation}
Now, we estimate the first term of the RHS of \eqref{expectation decomposition} when the event $\Omega_{D}$ takes place. Under this circumstance, we let
\begin{equation}\label{def.semi}
    f^*_{D,n}(\bx):=\sum_{i=1}^{|D|}a_{i,n,D}f^*(\mathbf x_i)
            K_n(\mathbf x_i\cdot \mathbf x).
\end{equation}
Let $\Lambda_D:=\{\mathbf x_i\}_{i=1}^{|D|}$. By the independence between $\{\epsilon_i\}_{i=1}^{|D|}$ and $
\Lambda_D$ and $\mathbf E\{\epsilon_i\}=0$,
$i=1,\dots, |D|$, we obtain
\begin{align*}
    \mathbf E\left\{f_{D,n}(\mathbf x)\big|\Lambda_{D}\right\}
     &= \mathbf E\left\{\sum_{i=1}^ma_{i,n,D}y_i K_n(\mathbf x_i\cdot \mathbf x)\big|\Lambda_{D}\right\}\\
     &= \mathbf E\left\{\sum_{i=1}^ma_{i,n,D}(f^*(\mathbf x_i)+\epsilon_i)K_n(\mathbf x_i\cdot \mathbf x)\big|\Lambda_{D}\right\}\\
     &= \sum_{i=1}^ma_{i,n,D}f^*(\mathbf x_i)K_n(\mathbf x_i\cdot \mathbf x)+\sum_{i=1}^ma_{i,n,D}\mathbf E\{\epsilon_i\}K_n(\mathbf x_i\cdot \mathbf x)\\[1mm]
     &= f^*_{D,n}(\mathbf x).
\end{align*}
Hence,
\begin{equation}\label{condition expectation}
     \mathbf E\left\{\left(f^*_{D,n}(\mathbf x)-f_{D,n}(\mathbf x)\right)\big|\Lambda_{D}\right\}=0.
\end{equation}
This allows us to write
\begin{align}\label{decomposition 1}
   &\mathbf E\left\{\|f_{D,n}-f^*\|^2_{L_2(\mathbb S^d)}\big|\Omega_{D}\right\}
     =
    \mathbf E\left\{\int_{\mathbb S^d}
     \mathbf E\{(f^*(\mathbf x)-f_{D,n}(\mathbf x))^2
    \big|\Lambda_{D}\}\mathrm{d}\omega(\mathbf x)\big|\Omega_{D}\right\} \nonumber\\
    &\quad=
    \mathbf E\left\{\int_{\mathbb S^d} \mathbf E\{(f^*(\mathbf x)-f^*_{D,n}(\mathbf x)
    +f^*_{D,n}(\mathbf x)-f_{D,n}(\mathbf x))^2
    \big|\Lambda_{D}\}\mathrm{d}\omega(\mathbf x)\big|\Omega_{D}\right\} \nonumber\\
    &\quad=
    \mathbf E\left\{ \int_{\mathbb S^d} \mathbf
    E\{(f^*_{D,n}(\mathbf x)-f_{D,n}(\mathbf x))^2\big|\Lambda_{D}\}\mathrm{d}\omega(\mathbf x)
    \big|\Omega_{D}\right\} \nonumber \\
        &\qquad+
         \mathbf E\left\{ \int_{\mathbb S^d} \mathbf
         E\{(f^*_{D,n}(\mathbf x)-f^*(\mathbf x))^2
         \big|\Lambda_{D}\}\mathrm{d}\omega(\mathbf x)\big|\Omega_{D}\right\}\nonumber\\
         &\quad :=
         \mathcal S_{D,n}+\mathcal A_{D,n}.
\end{align}
Given $\Lambda_{D}$, if the event  $\Omega_{D}$ occurs, by $|\epsilon_i|\leq M$,
\begin{align*}
          \mathbf
         E\left\{(f^*_{D,n}(\mathbf x)-f_{D,n}(\mathbf x))^2\big|\Lambda_{D}\right\}
         & =
          \mathbf E\left\{\left(\sum_{i=1}^{|D|}
          \epsilon_i a_{i,n,D}K_n(\mathbf x_i\cdot \mathbf x)\right)^2
          \bigg|\Lambda_{D}\right\}\\
          &\leq
           M^2  \sum_{i=1}^{|D|}a_{i,n,D}^2 |K_n(\mathbf x_i\cdot \mathbf x)|^2,
\end{align*}
where the second line uses the independence of $\epsilon_1,
\dots,\epsilon_{|D|}$.
This with Lemma~\ref{Lemma:bound for filtered norm} shows
\begin{align}\label{bound S}
        \mathcal S_{D,n}
         &\leq
         M^2\mathbf E\left\{ \int_{\mathbb S^d}  \sum_{i=1}^{|D|}
         a_{i,n,D}^2|K_n(\mathbf x_i\cdot \mathbf x)|^2
         \mathrm{d}\omega(\mathbf x) \big|\Omega_D\right\}\nonumber\\
          &=
          M^2\mathbf E\left\{\sum_{i=1}^{|D|} a_{i,n,D}^2\int_{\mathbb S^d}
          |K_n(\mathbf x_i\cdot \mathbf x)|^2 \mathrm{d}\omega(\mathbf x)\big|\Omega_D\right\} \nonumber \\
          &\leq
          c_1^2 M^2n^d\mathbf E\left\{\sum_{i=1}^{|D|} a_{i,n,D}^2\right\}\leq
         \frac{2c_1^2M^2n^d}{|D|}.
\end{align}

We now turn to bound $\mathcal A_{D,n}$. We split $\mathcal A_{D,n}$ as
\begin{align}\label{Decomposition A}
      \mathcal A_{D,n}
      &=
      \mathbf E\left\{\int_{\mathbb S^d} \mathbf E\left\{(f^*(\mathbf x)-f^*_{D,n}(\mathbf x))^2\big|\Lambda_{D}\right\}\mathrm{d}\omega(\mathbf x)
      \big|\Xi_D,\Omega_{D}\right\}\mathbf P\{\Xi_D\}\nonumber\\
    &\qquad+
     \mathbf E\left\{\int_{\mathbb S^d} \mathbf E\left\{(f^*(\mathbf x)-f^*_{D,n}(\mathbf x))^2\big|\Lambda_{D}\right\}\mathrm{d}\omega(\mathbf x)
     \big|\Xi_D^c,\Omega_{D}\right\}\mathbf P\{\Xi_D^c\} \nonumber\\
    &:=
    \mathcal A_{D,n,1}+\mathcal A_{D,n,2}.
\end{align}
To estimate $\mathcal A_{D,n,2}$ given the event $\Omega_{D}\cap\Xi_D^c$, it follows from Cauchy-Schwarz inequality that
\begin{align*}
     \bigl(f^*(\bx)-f^*_{D,n}(\bx)\bigr)^2
      &\leq 2\|f^*\|_{L_\infty(\mathbb S^d)}^2+2\left|\sum_{i=1}^{|D|}a_{i,n,D}f^*(\mathbf x_i)K_n(\mathbf x_i\cdot
       \bx)\right|^2\\
        &\leq
        2\|f^*\|_{L_\infty(\mathbb S^d)}^2+2\|f^*\|_{L_\infty(\mathbb S^d)}^2
        \sum_{i=1}^{|D|}a^2_{i,n,D}\sum_{i=1}^{|D|}|K_n(\mathbf x_i\cdot \bx)|^2,
\end{align*}
which with \eqref{Probability for events} and Lemma~\ref{Lemma:bound for filtered norm} gives
\begin{equation}\label{bound a2}
     \mathcal A_{D,n,2}\leq
     2\|f^*\|_{L_\infty(\mathbb S^d)}^2(|\sph{d}|+ 2c_1^2 n^{d})\exp\left\{-C|D|/n^d\right\}.
\end{equation}
To bound $\mathcal A_{D,n,1}$, we observe that when the event $\Omega_{D}\cap\Xi_D$ takes place, $\{a_{i,n,D}\}_{i=1}^{|D|}$ is a set of positive weights for quadrature rule $\mathcal Q_{|D|,n}$. We then obtain from Lemma~\ref{Lemma:Filtered h app} and $f^*\in \mathbb W_2^r(\mathbb S^d)$ with $r>d/2$ that
\begin{equation}\label{Bound a1}
      \mathcal A_{D,n,1}\leq c_5^2 n^{-2r}\|f\|^2_{\mathbb W_2^r(\mathbb S^d)}.
\end{equation}
 By
\eqref{Bound a1}, \eqref{bound a2} and \eqref{Decomposition A}, we obtain
\begin{equation}\label{bound a}
      \mathcal A_{D,n}\leq
      c_5^2n^{-2r}\|f^*\|^2_{\mathbb W_2^r(\mathbb
     S^d)}+ 2\|f^*\|_{L_\infty(\mathbb S^d)}^2(|\sph{d}|+ 2c_1^2 n^{d})\exp\{-C|D|/n^d\}.
\end{equation}
This and \eqref{bound S} and \eqref{decomposition 1} give
\begin{align*}%\label{eq:fDnerr}
      &\mathbf E\left\{\|f_{D,n}-f^*\|^2_{L_2(\mathbb S^d)}\big|\Omega_{D}\right\}\\
      &\;\leq
        c_5^2n^{-2r}\|f\|^2_{\mathbb W_2^r(\mathbb
     S^d)}+2\|f^*\|_{L_\infty(\mathbb S^d)}^2(|\sph{d}|+ 2c_1^2 n^{d})\exp\{-C|D|/n^d\}+\frac{2c_1^2M^2n^d}{|D|}.\notag
\end{align*}
Putting the above estimate and \eqref{bound without event} into
\eqref{expectation decomposition}, we obtain
\begin{align}
    \mathbf E\left\{\|f_{D,n}-f^*\|^2_{L_2(\mathbb S^d)}\right\}
    &\leq
     c_5^2n^{-2r}\|f\|^2_{\mathbb W_2^r(\mathbb
     S^d)}+\frac{2c_1^2M^2n^d}{|D|}\notag\\
    &\qquad + 2\|f^*\|_{L_\infty(\mathbb S^d)}^2(|\sph{d}|+ 2c_1^2 n^{d}+2)\exp\bigl\{-C|D|/n^d\bigr\}.\label{IMportant}
\end{align}
Taking account of $\frac{c_3}6|D|^{\frac1{2r+d}}\leq n\leq \frac{c_3}3|D|^{\frac1{2r+d}}$ and $r>d/2$, we then have
\begin{equation*}
	n^{d}\exp\left\{-C|D|/n^d\right\}
	\leq \left(\frac{c_3}{3}\right)^d|D|^{\frac{d}{2r+d}}\exp\left\{-C|D|^{\frac{2r}{2r+d}}\right\}\leq\tilde{C}_5|D|^{-\frac{2r}{2r+d}},
\end{equation*}
where $\tilde{C}_5$ is a constant independent of $|D|$. Thus,
\begin{equation*}
	\mathbf E\left\{\|f_{D,n}-f^*\|^2_{L_2(\mathbb S^d)}\right\}\leq C_3|D|^{-\frac{2r}{2r+d}}
\end{equation*}
with $C_3$ a constant independent of $|D|$, thus completing
the proof.
\end{proof}

To prove Theorem~\ref{Theorem:distributed learning rate}, we need
the following lemma, which can be obtained by a similar proof
as Lemma~\ref{Lemma:distributed.1}.
\begin{lemma}\label{Lemma:distributed}
 For $\overline{f}_{D,n}$ in Definition~\ref{defn:distrfih.ransamp}, there holds
\begin{align*}
    \mathbf E\left\{\|\overline{f}_{D,n}-f^*\|_{L_2(\mathbb S^d)}^2\right\}
    &\leq \sum_{j=1}^m\frac{|D_j|^2}{|D|^2} \mathbf E\left\{\|f_{D_j,n}-f^*\|_{L_2(\mathbb S^d)}^2\right\}\\
    &\qquad+\sum_{j=1}^m\frac{|D_j|}{|D|} \left\|\mathbf E\{f_{D_j,n}\}-f^*\right\|_{L_2(\mathbb S^d)}^2.
\end{align*}
\end{lemma}

\begin{proof}[Proof of Theorem~\ref{Theorem:distributed learning
rate}] By Lemma~\ref{Lemma:distributed}, we only
need to estimate the bounds of $\mathbf E\bigl\{\|f_{D_j,n}-f^*\|_{L_2(\mathbb S^d)}^2\bigr\}$ and
         $\left\|\mathbf E\{f_{D_j,n}\}-f^*\right\|_{L_2(\mathbb
         S^d)}^2$.
To estimate the first, we obtain from \eqref{IMportant} with $D=D_j$
that for $j=1,\dots,m$,
\begin{align*}
    \mathbf E\left\{\|f_{D_j,n}-f^*\|^2_{L_2(\mathbb S^d)}\right\}
    &\leq
     c_5^2n^{-2r}\|f\|^2_{\mathbb W_2^r(\mathbb
     S^d)} + \frac{2c_1^2M^2n^d}{|D_j|}\notag\\
    &\quad + 2\|f^*\|_{L_\infty(\mathbb S^d)}^2\left(|\sph{d}|+ 2c_1^2 n^{d}+2\right)
     \exp\bigl\{-C|D_j|/n^d\bigr\}.
\end{align*}
Since $\min_{1\leq j\leq m}|D_j|\geq |D|^\frac{d+\nu}{2r+d}$, $\frac{c_3}6|D|^{\frac{1}{2r+d}}\leq n\leq \frac{c_3}3|D|^{\frac{1}{2r+d}}$, $2r>d$ and $0<\nu<2r$,
\begin{equation*}
	2\|f^*\|_{L_\infty(\mathbb S^d)}^2\left(|\sph{d}|+ 2c_1^2 n^{d}+2\right) \exp\bigl\{-C|D_j|/n^d\bigr\}\leq
      \tilde{C}_7|D|^{-\frac{2r}{2r+d}},
\end{equation*}
where $\tilde{C}_7$ is a constant depending only on $r,c_1,C,d$ and
$f^*$. Thus, there exists a constant $\tilde{C}_8$ independent of
$m, n, |D_1|,\dots,|D_m|$ and $|D|$ such that
\begin{align}\label{bound first term}
         &\sum_{j=1}^m\frac{|D_j|^2}{|D|^2}
         \mathbf E\left\{\|f_{D_j,n}-f^*\|_{L_2(\mathbb
         S^d)}^2\right\}  \nonumber\\
         &\qquad\leq
        \tilde{C}_8 \left(|D|^{-\frac{2r}{2r+d}}+
        \sum_{j=1}^m\frac{|D_j|^2}{|D|^2}\frac{|D|^{\frac{d}{2r+d}}}{|D_j|}\right)
        = (\tilde{C}_8+1)|D|^{-\frac{2r}{2r+d}},
\end{align}
where we used $\sum_{i=1}^m \frac{|D_j|}{|D|}=1$. To bound $\left\|\mathbf E\{f_{D_j,n}\}-f^*\right\|_{L_2(\mathbb S^d)}^2$, let $\Lambda_{D_j}$ be the set of points of the data set $D_j$. Then, we use \eqref{condition expectation} and Jensen's inequality to obtain
\begin{align}\label{eq:err_fDjn_f*}
    \left\|\mathbf E\{f_{D_j,n}\}-f^*\right\|_{L_2(\mathbb S^d)}^2
    &= \left\|\mathbf E\{\mathbf E\{f_{D_j,n}|\Lambda_{D_j}\}-f^*\}\right\|_{L_2(\mathbb S^d)}^2\notag\\
    &= \left\|\mathbf E\{f^*_{D_j,n}-f^*\}\right\|_{L_2(\mathbb S^d)}^2
    \leq \mathbf E\left\{\|f^*_{D_j,n}-f^* \|_{L_2(\mathbb S^d)}^2\right\}.
\end{align}
We now use a similar proof as Theorem~\ref{Theorem:learning rate fih random samp} to prove the error bound of distributed filtered hyperinterpolation $\overline{f}_{D,n}$. For each $j=1,\dots,m$, we let $\Omega_{D_j}$ be the event such that the sum of the quadrature weights $\sum_{i=1}a_{i,n,D_j}^2\leq 2/|D_{j}|$, and $\Omega_{D_j}^c$ the complement of $\Omega_{D_j}$. Write 
\begin{align}
      \mathbf E\left\{\|f^*_{D_j,n}-f^*\|^2_{L_2(\mathbb S^d)}\right\}
      &=\mathbf E\left\{\|f^*_{D_j,n}-f^*\|^2_{L_2(\mathbb S^d)}|\Omega_{D_j}\right\}\mathbf P\{\Omega_{D_j}\}\notag\\
      &\qquad+\mathbf E\left\{\|f^*_{D_j,n}-f^*\|^2_{L_2(\mathbb S^d)}|\Omega_{D_j}^c\right\}\mathbf P\{\Omega_{D_j}^c\},\label{eq:f*Dj}
\end{align}
where,
\begin{equation*}
	\mathbf E\left\{\|f^*_{D_j,n}-f^*\|^2_{L_2(\mathbb S^d)}|\Omega_{D_j}^c\right\}\mathbf P\{\Omega_{D_j}^c\}
      \leq 4\|f^*\|_{L_\infty(\mathbb S^d)}^2\exp\bigl\{-C|D_j|/n^d\bigr\}.
\end{equation*}
By \eqref{bound a} with $D=D_j$, the second term of the RHS in \eqref{eq:f*Dj} becomes
\begin{align*}
     &\mathbf E\left\{\|f^*_{D_j,n}-f^*\|^2_{L_2(\mathbb S^d)}
      |\Omega_{D_j}\right\}\mathbf P\{\Omega_{D_j}\} \\
     &\quad\leq
       c_5^2n^{-2r}\|f\|^2_{\mathbb W_2^r(\mathbb
     S^d)}+ 2\|f^*\|_{L_\infty(\mathbb S^d)}^2(|\sph{d}|+ 2c_1^2 n^{d})\exp\bigl\{-C|D_j|/n^d\bigr\}.
\end{align*}
These two estimates give
\begin{align*}
    &\mathbf E\left\{\|f^*_{D_j,n}-f^*\|^2_{L_2(\mathbb S^d)}\right\}\\
    &\quad\leq c_5^2n^{-2r}\|f\|^2_{\mathbb W_2^r(\mathbb
     S^d)} + 2\|f^*\|_{L_\infty(\mathbb S^d)}^2(|\sph{d}|
     +2c_1^2n^{d}+2)\exp\bigl\{-C|D_j|/n^d\bigr\}.
\end{align*}
By $\min_{1\leq j\leq m}|D_j|\geq |D|^\frac{d+\nu}{2r+d}$, $\frac{c_3}6 |D|^{\frac{1}{2r+d}} \leq n\leq \frac{c_3}3 |D|^{\frac{1}{2r+d}}$ and $2r>d$, $0<\nu<2r$,
\begin{equation*}
	\mathbf E\left\{\|f^*_{D_j,n}-f^*\|^2_{L_2(\mathbb S^d)}\right\}
        \leq \tilde{C}_9|D|^{-\frac{2r}{2r+d}},
\end{equation*}
which with \eqref{eq:err_fDjn_f*} and $\sum_{j=1}^m\frac{|D_j|}{|D|}=1$ gives
\begin{equation}\label{bound second term}
     \sum_{j=1}^m\frac{|D_j|}{|D|}\left\|\mathbf E\{f_{D_j,n}\}-f^*\right\|_{L_2(\mathbb S^d)}^2
     \leq \tilde{C}_{9}|D|^{-\frac{2r}{2r+d}}.
\end{equation}
Using \eqref{bound first term} and \eqref{bound second term} in
Lemma~\ref{Lemma:distributed} then gives
$$
     \mathbf E\left\{\|\overline{f}_{D,n}-f^*\|_{L_2(\mathbb S^d)}^2\right\}
     \leq
     C_4|D|^{-\frac{2r}{2r+d}},
$$
thus completing the proof.
\end{proof}

%\section*{Acknowledgments}

{\rm \bibliographystyle{siamplain}
\bibliography{dlfh}}
\end{document}